\documentclass[12pt]{amsart}
\usepackage[all]{xy}
\usepackage[parfill]{parskip}

\usepackage{amsfonts}
\usepackage{mathrsfs}
\usepackage[abs]{overpic}

\usepackage{palatino} 

\usepackage{hyperref}

\usepackage{verbatim}
\usepackage{color}

\usepackage{amsmath, amscd, graphicx, latexsym, hyperref, times, calc}
\usepackage{color,graphicx}
\usepackage[abs]{overpic}
\usepackage{tikz}
\usepackage{tikz-cd}
\usetikzlibrary{arrows, patterns}

\newtheorem{dummy}{anything}[section]
\newtheorem{theorem}[dummy]{Theorem}

\newtheorem{lemma}[dummy]{Lemma}

\newtheorem{question}[dummy]{Question}

\theoremstyle{definition}

\newtheorem{remark}[dummy]{Remark}

\newtheorem*{acknowledgements}{Acknowledgements}

\usepackage[utf8]{inputenc}

\title{On geography of symplectic fillings of contact branched covers}
\author{Youlin Li and Yuhe Zhang}

\address{School of Mathematical Sciences, Shanghai Jiao Tong University, Shanghai 200240, China}
\email{liyoulin@sjtu.edu.cn}

\address{School of Mathematical Sciences, Shanghai Jiao Tong University, Shanghai 200240, China}
\email{yuhe\_tmp@sjtu.edu.cn}

\begin{document}

\maketitle

\begin{abstract}
In this paper, we determine the Euler characteristics and signatures of the exact symplectic fillings of the contact double, 3-fold or 4-fold cyclic covers of the standard contact 3-sphere branched over certain transverse quasi-positive links. These links include all quasi-positive knots with crossing numbers smaller than 11 and all quasi-positive links with crossing numbers smaller than 12 and nonzero nullity.
\end{abstract}

\section{Introduction}
Given a closed contact 3-manifold, understanding its symplectic fillings is an important topic in contact topology. This topic was intensively studied in the last two decades. There are plenty of contact 3-manifolds whose symplectic fillings have been classified up to certain equivalence relations. For example the standard contact 3-sphere \cite{eli},  the  contact lens spaces \cite{m, l1, er, cl}, etc. In \cite{s}, Stipsicz raised the geography problem for Stein fillings: Fix a contact 3-manifold $(M,\xi)$ and describe characteristic numbers (first Betti numbers, Euler characteristic numbers and signatures) of Stein fillings of it. Since then, the geography for symplectic fillings of many contact 3-manifolds have been figured out, even though their symplectic fillings have not been classified yet. For example some contact Brieskorn spheres \cite{s, s1, eg},   the contact 3-manifolds supported by planar (spinal) open book decompositions \cite{e1, k1, lw}, the contact 3-manifolds satisfying certain Floer homology conditions \cite{os, lin}, and the contact 3-manifolds wearing certain symplectic caps \cite{lmy, lma}, etc.

Given a transverse link $L$ in the standard contact 3-sphere $(S^3, \xi_{st})$, there is a canonical way to construct a contact structure on the manifold which is a branched covering of $S^3$ over $L$ \cite{g}. In fact, any closed contact 3-manifold can be described as such a contact cover of $(S^3, \xi_{st})$ branched over some transverse link \cite{gi}. In this paper, we consider the exact symplectic fillings of the contact cyclic covers of $(S^3, \xi_{st})$ branched over transverse links. A transverse link in $(S^3, \xi_{st})$ can be represented by a closed braid, and vice versa \cite{b,e}. An $n$-string braid is called quasi-positive if it can be presented by $\prod\limits_{i=1}^{m} w_{i}\sigma_{j_i}w^{-1}_{i}$, where $\sigma_{j_i}$ is a standard generator of the $n$-string braid group, and $w_i$ is any word in the $n$-string braid group. We call $w_{i}\sigma_{j_i} w^{-1}_{i}$ a positive band.  An oriented link is quasi-positive if it is the closure of a quasi-positive braid. By \cite[Proposition 1.4]{p} and \cite[Theorem 1.3]{hkp}, a contact cyclic cover of $(S^3, \xi_{st})$ branched over a quasi-positive transverse link is Stein fillable.

For a transverse link $L$ in $(S^3, \xi_{st})$ and a symplectic cap $C$ of $(S^3, \xi_{st})$, Etnyre and Golla \cite{eg} defined the symplectic hat of $L$ to be a properly embedded symplectic surface $\hat{F}$ in $C$ with $\partial \hat{F}=-L$. They studied various kinds of special symplectic hats \cite{eg}, among which projective hats and Hirzebruch hats are useful in the study of symplectic fillings. A symplectic hat of a transverse link in $(S^3, \xi_{st})$ in the projective cap $\mathbb{C}P^{2}-\text{Int}(B^4)$  (Hirzebruch cap $\mathbb{C}P^{1}\times \mathbb{C}P^{1}-\text{Int}(B^4)$, resp.)  is called a projective hat (Hirzebruch hat, resp.). A projective hat is of degree $d$ if its intersection number with the line at infinity $l_{\infty}$ is $d$.  A Hirzebruch hat is of bidegree $(d_1, d_2)$ if its intersetion numbers with the two  $\mathbb{C}P^{1}$ factors are $d_1$ and $d_2$ respectively. 

Li, Mak and Yasui defined and studied Calabi-Yau caps for contact 3-manifolds \cite{lmy}, which put strong constraints on the topology of exact symplectic fillings.  A Calabi-Yau cap for a contact 3-manifold is a symplectic cap with torsion first Chern class.

In \cite{eg}, utilizing the projective hats and Hirzebruch hats, Etnyre and Golla constructed Calabi-Yau caps and then determined the intersection forms of exact symplectic fillings of contact double covers of $(S^3, \xi_{st})$ branched over all determinant-$1$ quasi-positive transverse knots which have crossing numbers $\leq 12$ and contact double, 3-fold and 4-fold cyclic covers of $(S^3, \xi_{st})$ branched over all slice quasi-positive transverse knots which have crossing numbers $\leq 12$.

In this paper, we determine the geography of the exact fillings of the contact double, 3-fold and 4-fold cyclic cover of $(S^3,\xi_{st})$ branched over the quasi-positive transverse links under some assumptions.  We denote by $\Sigma_{r}(L)$ the contact $r$-fold cyclic cover of $(S^3, \xi_{st})$ branched over a transverse link $L$.

\begin{theorem}\label{rfold}
Suppose $L$ is a quasi-positive transverse link in $(S^3, \xi_{st})$, and the braid representing $L$ is of $n$-string and has $m$ bandwords.  If  
\begin{itemize}
\item[(1)] $L$ wears a projective hat of degree $6$ and $r=2$, or
\item[(2)] $L$ wears a Hirzebruch hat of bidegree $(3,3)$ and $r=3$, or
\item[(3)] $L$ wears a projective hat of degree $4$ and $r=4$,
\end{itemize}
then any exact filling $W$ of $\Sigma_{r}(L)$ is spin and has 
$$\chi(W)=r-(r-1)(n-m)~ \text{and}~ \sigma(W)= \sum\limits_{k=1}^{r-1}\sigma_{L}(\zeta^{k}),$$
provided that 
\begin{equation*}
  (r-1)(1-n+m)+2\sum\limits_{k=1}^{r-1}\eta_{L}(\zeta^{k})+\sum\limits_{k=1}^{r-1}\sigma_{L}(\zeta^{k})\leq 3,  
\end{equation*}
and
\begin{equation*}
(r-1)(1-n+m)+2\sum\limits_{k=1}^{r-1}\eta_{L}(\zeta^{k})-\sum\limits_{k=1}^{r-1}\sigma_{L}(\zeta^{k})\leq 25.    
\end{equation*}
where $\sigma_{L}:S^{1}\rightarrow \mathbb{Z}$ is the Levine-Tristram signature function of $L$, $\eta_{L}:S^{1}\rightarrow \mathbb{Z}$ is the Levine-Tristram nullity function of $L$, and $\zeta=e^{\frac{2\pi i}{r}}$.
\end{theorem}

\begin{theorem}\label{rfoldcover}
Suppose $L$ is a quasi-positive transverse link in $(S^3, \xi_{st})$ with $\sum\limits_{k=1}^{r-1}\eta_{L}(\zeta^{k})=0$, and the braid representing $L$ is of $n$-string and has $m$ bandwords.  If 
\begin{itemize}
\item[(1)] $L$ wears a projective hat of degree $6$  and $r=2$, or
\item[(2)] $L$ wears a Hirzebruch hat of bidegree $(3,3)$ and $r=3$, or
\item[(3)] $L$ wears a projective hat of degree $4$ and $r=4$,
\end{itemize}
then any exact filling $W$ of $\Sigma_{r}(L)$ is spin and has  $$b_{1}(W)=0,
\chi(W)=r-(r-1)(n-m)~\text{and} ~ \sigma(W)=\sum\limits_{k=1}^{r-1}\sigma_{L}(\zeta^{k})$$
provided that 
\begin{equation*}
(r-1)(1-n+m)+\sum\limits_{k=1}^{r-1}\sigma_{L}(\zeta^{k})\leq2, \end{equation*}
and
\begin{equation*}
(r-1)(1-n+m)\leq15.  
\end{equation*}
\end{theorem}

\begin{remark}
If $L$ is a transverse knot,  then its self-linking number $sl(L)=m-n$.
\end{remark}

As applications, we consider the cyclic cover of $(S^{3}, \xi_{st})$ branched over quasi-positive transverse knots with crossing number $\leq 10$. We assume that their transverse representations are all come from the closure of the quasi-positive braids appeared in Knotinfo \cite{lm}.  Among these knots, $m(8_{20})$, $m(9_{46})$, $10_{124}$, $10_{140}$ and $m(10_{155})$ have been fully treated, and $8_{21}$ has been partially treated, by Etnyre and Golla in \cite{eg}. 

\begin{theorem}\label{2foldcover}
Any exact filling $W$ of $\Sigma_{2}(K)$ is spin, has $b_1(W)=0$, 

(1) $\chi(W)=3$ and $\sigma(W)=-2$ if $K$ is $3_1$, $5_2$, $7_2$, $7_4$, $8_{21}$, $9_{2}$, $9_{5}$, $9_{35}$, $m(9_{45})$, $m(10_{126})$, $10_{131}$, $10_{133}$, $m(10_{143})$, $m(10_{148})$, $10_{159}$ or $m(10_{165})$. 

(2) $\chi(W)=5$ and $\sigma(W)=-4$ if $K$ is  $5_1$,  $7_3$, $7_5$, $8_{15}$, $9_{4}$, $9_{7}$, $9_{10}$, $9_{13}$, $9_{18}$, $9_{23}$, $9_{38}$, $9_{49}$, $10_{53}$, $10_{55}$, $10_{63}$, $10_{101}$, $10_{120}$, $10_{127}$, $10_{149}$ or $m(10_{157})$, and $\sigma(W)=-2$  if $K$ is $m(10_{145})$.  

(3) $\chi(W)=7$ and $\sigma(W)=-6$ if $K$ is  $7_1$,  $8_{19}$, $9_3$, $9_6$, $9_{9}$, $9_{16}$, $10_{49}$,  $10_{66}$, $10_{80}$, $10_{128}$, $10_{134}$ or $10_{142}$, and $\sigma(W)=-4$ if $K$ is either $10_{154}$ or $10_{161}$.  

(4) $\chi(W)=9$ and $\sigma(W)=-8$ if $K$ is $9_{1}$, 
and $\sigma(W)=-6$ if $K$ is  $10_{139}$ or $10_{152}$. 

\end{theorem}

\begin{theorem}\label{3foldcover}
Any exact filling $W$ of $\Sigma_{3}(K)$ has $b_1(W)=0$,

(1) $\chi(W)=5$ and $\sigma(W)=-4$ if $K$ is  $3_1$, $5_{2}$, $7_{2}$, $7_{4}$, $8_{21}$, $9_{2}$, $9_{5}$, $9_{35}$, $m(9_{45})$, $m(10_{126})$, $10_{131}$, $m(10_{143})$, $m(10_{148})$, $10_{159}$ or $m(10_{165})$.

(2) $\chi(W)=9$ and $\sigma(W)=-8$ if $K$ is  $5_1$, $7_3$, $7_5$, $8_{15}$, $9_4$, $9_7$, $9_{10}$, $9_{13}$, $9_{18}$, $9_{23}$, $9_{38}$, $9_{49}$, $10_{120}$, $10_{127}$, $10_{149}$ or $m(10_{157})$.

(3) $\chi(W)=13$ and $\sigma(W)=-12$ if $K$ is  $9_6$ or $9_9$.

(4) $\chi(W)=13$ and $\sigma(W)=-8$ or is negative definite if $K$ is  $7_1$, $8_{19}$, $9_3$, $10_{128}$ or $10_{134}$.
\end{theorem}

\begin{theorem}\label{4foldcover}
Any exact filling $W$ of $\Sigma_{4}(K)$ has $b_1(W)=0$,

(1) $\chi(W)=7$ and $\sigma(W)=-6$ if $K$ is  $3_1$, $5_2$, $7_2$, $7_4$, $8_{21}$, $m(9_{45})$,  $m(10_{126})$, $m(10_{143})$, $m(10_{148})$ or $10_{159}$.

(2) $\chi(W)=13$ and $\sigma(W)=-12$ if $K$ is  $7_5$ or $8_{15}$.

(3) $\chi(W)=13$ and $\sigma(W)=-8$ or is negative definite if $K$ is $5_1$ or $7_{3}$.

\end{theorem}

\begin{remark}
(1) By \cite[Lemma 6.6]{eo}, the Euler classes of all these contact double branched covers are $0$. The double branched cover of $S^3$ over the following knots are lens spaces: $3_1$, $5_1$, $5_2$, $7_1$, $7_2$, $7_3$, $7_4$, $7_5$, $9_1$, $9_2$, $9_3$, $9_4$, $9_5$, $9_6$, $9_7$, $9_9$, $9_{10}$, $9_{13}$, $9_{18}$ and $9_{23}$. In \cite[Page 439]{ef}, the tight contact structure on a lens space with vanishing Euler class is identified with the standard Legendrian surgery diagram in \cite{h} whose all components have rotation number $0$. The symplectic fillings of such contact lens spaces has been classified, cf. \cite{l1, er, cl}. 


(2) The cyclic branched covering $\Sigma_{3}(5_1)$ is the Poincare homology sphere $\Sigma(2,3,5)$. The canonical contact structure on $\Sigma(2,3,5)$ has a Stein filling with intersection form $E_8$, the unique negative definite, even, unimodular form of rank $8$ \cite{s}. 

(3) The cyclic branched covering $\Sigma_{3}(7_1)$ is the Brieskorn sphere $\Sigma(2,3,7)$. According to \cite[5.3.3]{ns}, $\Sigma_{4}(5_1)$ is the Brieskorn sphere $\Sigma(2,4,5)$ that is the Seifert fibered space $M(-1;\frac{1}{2}, \frac{1}{5}, \frac{1}{5})$. The exact fillings of the canonical contact structures on these two Brieskorn spheres have been investigated by Etnyre and Golla in \cite{eg}.   


(4)  By \cite[5.3.3]{ns}, $\Sigma_{3}(3_1)$ is the Brieskorn sphere $\Sigma(2,3,3)$ that is the Seifert fibered space $M(-2;\frac{1}{2}, \frac{1}{2}, \frac{1}{2})$, and both  $\Sigma_{2}(8_{19})$ and $\Sigma_{4}(3_{1})$ are the Brieskorn sphere $\Sigma(2,3,4)$ that is the Seifert fibered space $M(-2;\frac{1}{2}, \frac{2}{3}, \frac{2}{3})$. According to the calssification result in \cite{t}, the canonical contact structure on $\Sigma(2,3,3)$ has a negative definite exact filling with $b_{2}=4$, and the canonical contact structure on $\Sigma(2,3,4)$ has an exact filling with intersection form negative definite $E_6$. Moreover,  $\Sigma_{3}(8_{19})$ is the Brieskorn sphere $\Sigma(3,3,4)$ that is the Seifert fibered space $M(-1;\frac{1}{4}, \frac{1}{4}, \frac{1}{4})$. 
\end{remark}


Similarly, one can determine the geography of the exact fillings of the contact cyclic coverings of $(S^3,\xi_{st})$ branched over quasi-positive knots with more crossings and quasi-positive links with zero nullity. In the following, we consider all quasi-positive transverse links with crossing numbers $\leq 11$ and nonzero nullity. Unlike the previous cases, the double branched covers of these links are not rational homology spheres.

\begin{theorem}\label{2foldcoverlink}
Any exact filling $W$ of $\Sigma_{2}(L)$ is spin, has 

(1) $\chi(W)=0$ and $\sigma(W)=0$ if $L$ is $m(L10n104\{1,0,0\})$ or $m(L10n104\{1,1,0\})$. In fact, $b_{1}(W)=1$ and $b_{2}(W)=0$.

(2) $\chi(W)=1$ and $\sigma(W)=-1$ if $L$ is  $L10n94\{1,0\}$, $m(L11n381\{0,0\})$, $m(L11n381\{0,1\})$ or $L11n428\{1,0\}$. In fact, $b_{1}(W)=1$, $b^{+}_{2}(W)=0$, $b^{-}_{2}(W)=1$ and $b^{0}_{2}(W)=0$.

(3) $\chi(W)=2$ and $\sigma(W)=-2$ if $L$ is  $m(L11n226\{0\})$. In fact, $b_{1}(W)=1$, $b^{+}_{2}(W)=0$, $b^{-}_{2}(W)=2$ and $b^{0}_{2}(W)=0$.

(4) $\chi(W)=2$ and $\sigma(W)=0$ if $L$ is $L11n204\{1\}$.

(5) $\chi(W)=3$ and $\sigma(W)=-1$ if $L$ is $L10n93\{0,1\}$ or $L11n379\{1,1\}$.

(6) $\chi(W)=4$ and $\sigma(W)=-2$ if $L$ is  $L9n18\{1\}$, $m(L10n104\{0,0,0\})$ or $m(L11n205\{0\})$.

(7) $\chi(W)=5$ and $\sigma(W)=-3$ if $L$ is  $m(L8n6\{0,0\})$, $L10n94\{0,0\}$ or $m(L11n433\{0,1\})$.

(8) $\chi(W)=6$ and $\sigma(W)=-4$ if $L$ is  $L8n8\{1,0,1\}$,  $m(L9n19\{0\})$, $L9n19\{1\}$, \\ $m(L10n111\{0,0,1\})$ or $L11n237\{1\}$.

(9) $\chi(W)=7$ and $\sigma(W)=-5$ if $L$ is  $L8n6\{1,0\}$, $L10n91\{1,0\}$, $m(L10n94\{0,1\})$,\\ $m(L11n411\{0,0\})$, $m(L11n432\{0,1\})$ or $L11n437\{1,0\}$.

(10) $\chi(W)=8$ and $\sigma(W)=-6$ if $L$ is  $m(L9n18\{0\})$, $L10n104\{1,0,1\}$, $L11n205\{1\}$, $m(L11n236\{0\})$ or $L11n459\{1,1,1\}$.

(11) $\chi(W)=9$ and $\sigma(W)=-7$ if $L$ is  $m(L10n93\{0,0\})$ or $m(L11n379\{0,1\})$.

(12) $\chi(W)=10$ and $\sigma(W)=-8$ if $L$ is $m(L11n204\{0\})$.

\end{theorem}

\begin{remark}
In the cases (4)-(12), the first Betti number $b_{1}(W)$ is either $0$ or $1$.
\end{remark}

\begin{theorem}\label{3foldcoverlink}
Any exact filling $W$ of $\Sigma_{3}(L)$ is spin, has 

(1)  $\chi(W)=-1$ and $\sigma(W)=0$ if $L$ is  $m(L10n104\{1,0,0\})$ or $m(L10n104\{1,1,0\})$.  In fact, $b_{1}(W)=2$ and $b_{2}(W)=0$.

(2)  $\chi(W)=1$ and $\sigma(W)=-2$ if $L$ is  $m(L11n381\{0,0\})$,  $m(L11n381\{0,1\})$ or $L11n428\{1,0\}$. In fact, $b_{1}(W)=2$, $b^{+}_{2}(W)=0$, $b^{-}_{2}(W)=2$ and $b^{0}_{2}(W)=0$.

(3)  $\chi(W)=1$ and $\sigma(W)=0$ if $L$ is $L10n94\{1,0\}$.  In fact, $b_{1}(W)=0$ and $b_{2}(W)=0$.

(4)  $\chi(W)=3$ and $\sigma(W)=-2$ if $L$ is $m(L11n226\{0\})$.  In fact, $b_{1}(W)=0$, $b^{+}_{2}(W)=0$, $b^{-}_{2}(W)=2$ and $b^{0}_{2}(W)=0$.
\end{theorem}

\begin{theorem}\label{4foldcoverlink}
Any exact filling $W$ of $\Sigma_{4}(L)$ is spin, has 

(1)  $\chi(W)=-2$ and $\sigma(W)=0$ if $L$ is  $m(L10n104\{1,0,0\})$ or $m(L10n104\{1,1,0\})$.  In fact, $b_{1}(W)=3$ and $b_{2}(W)=0$.

(2)  $\chi(W)=1$ and $\sigma(W)=-3$ if $L$ is  $m(L11n381\{0,0\})$, $m(L11n381\{0,1\})$ or $L11n428\{1,0\}$. In fact, $b_{1}(W)=3$, $b^{+}_{2}(W)=0$, $b^{-}_{2}(W)=3$ and $b^{0}_{2}(W)=0$.

(3)  $\chi(W)=1$ and $\sigma(W)=-1$ if $L$ is $L10n94\{1,0\}$. In fact, $b_{1}(W)=1$, $b^{+}_{2}(W)=0$, $b^{-}_{2}(W)=1$ and $b^{0}_{2}(W)=0$.

(4)  $\chi(W)=4$ and $\sigma(W)=-4$ if $L$ is  $m(L11n226\{0\})$.  In fact, $b_{1}(W)=1$, $b^{+}_{2}(W)=0$, $b^{-}_{2}(W)=4$ and $b^{0}_{2}(W)=0$.
\end{theorem}

In the proof of above theorems, we always construct a relative symplectic cobordism from a transverse link to some transverse torus knot. A transverse link $L_{-}$ is relatively symplectic cobordant to a transverse link $L_{+}$ in $(S^3, \xi_{st}=\ker\alpha_{st})$ if there exists a properly embedded symplectic surface $\Sigma$ in $(S^3\times [0,1], d(e^{t}\alpha_{st}))$ such that $L_{-}\subset S^3\times \{0\}$, $L_{+}\subset S^3\times \{1\}$ and $\partial \Sigma=-L_{-}\sqcup L_{+}$. See \cite{eg}.  It is natural to ask the following question.

\begin{question}
Given a transverse knot $K$ in $(S^3, \xi_{st})$. Do there exist infinitely many transverse quasi-positive knots (or nonsplittable links) which are all relatively symplectic cobordant to $K$?
\end{question}

\medskip
\begin{acknowledgements}
The authors would like to thank Ying Chen, Anthony Conway, John Etnyre and Langte Ma for helpful correspondence. The authors were partially supported by Grant No. 11871332 of the National Natural Science Foundation of China.
\end{acknowledgements}

\section{Exact symplectic fillings}

At first, we recall the definition and properties of the Levine-Tristram signature and nullity. Let $L$ be an oriented link, let $F$ be a Seifert surface for $L$ with $b_{0}(F)$
components and let $A$ be a matrix representing the Seifert pairing of $F$. Given $\omega\in 
S^1$, the Levine-Tristram signature and nullity of $L$ at $\omega$ are defined in \cite{l, t1} as
$$\sigma_{L}(\omega):=\text{signature}((1-\omega)A + (1-\overline{\omega})A^{T}),$$
$$\eta_{L}(\omega):=\text{nullity}((1-\omega)A + (1-\overline{\omega})A^{T})+b_{0}(F)-1.$$
We list some well-known properties of $\sigma_{L}(\omega)$ and $\eta_{L}(\omega)$ as in the following lemma.

\begin{lemma} \label{sig0}\cite{bbg, c}
(1) $\sigma_{L}(-1)$ is the Murasugi signature $\sigma(L)$, and $\eta_{L}(-1)$ is the classical nullity $\eta(L)$.

(2) $\sigma_{L}(\omega)=\sigma_{L}(\bar{\omega})$ and $\eta_{L}(\omega)=\eta_{L}(\bar{\omega})$ for all $\omega\in S^{1}$.

(3) $\sigma_{m(L)}(\omega)=-\sigma_{L}(\omega)$ and $\eta_{m(L)}(\omega)=\eta_{L}(\omega)$, where $m(L)$ is the mirror image of $L$.

(4) If $K'$ is a satellite knot with companion knot $K$, pattern $P$ and wrapping number $q$, then $\sigma_{K'}(\omega)=\sigma_{P}(\omega)+\sigma_{K}(\omega^q)$.

(5) The first Betti number $b_{1}(\Sigma_{r}(L))=\sum\limits_{k=1}^{r-1}\eta_{L}(\zeta^{k}),$ where $\zeta$ is an $r$-th primitive root of unity. 
\end{lemma}

One can find more details about these two invariants in \cite{bbg, c}. Anthony Conway has the following observation.

\begin{lemma}\label{sig1}\cite{co}
Suppose that an oriented link in $S^3$ bounds a smooth, properly embedded, oriented compact surface $\check{F}$ in $B^4$. Let $\Sigma_{r}(\check{F})$ be the $r$-fold cyclic cover of $B^4$ branched over $\check{F}$. Then the signature of $\Sigma_{r}(\check{F})$ is $\sum\limits_{k=1}^{r-1}\sigma_{L}(\zeta^{k})$.
\end{lemma}

\begin{proof}
The $\mathbb{C}$-vector space $V=H_2(\Sigma_{r}(\check{F}),\mathbb{C})$ is acted upon by the covering transformation $\tau$. Since $\tau^r=1$, we can decompose $V$ as $V=V_0\oplus V_1\oplus \cdots \oplus V_{r-1}$, where $V_k$ denotes the $\zeta^k$-eigenspace $\{x\in V \mid \tau x= \zeta^{k}x\}$. So the signature of $V$ decomposes as a sum of the signature of $V_k$. It follows from Viro's theorem \cite{v} that the signature of $V_k$ equals $\sigma_L(\zeta^k)$ for $k>0$. By the same argument as in \cite[Page 333]{k2}, the signature of $V_0$ equals the signature of $B^4$ which vanishes. 
\end{proof}

Now we prove the results.

\begin{proof}[Proof of Theorem~\ref{rfold}]
According to the main result in \cite{r}, there is a complex analytic curve, and hence a symplectic surface, $\check{F}$ in $B^4$ whose boundary is $L$. By \cite[Corollary 4.83]{r1}, $\chi(\check{F})=n-m$ if $L$ is the closure of an $n$-string braid $\prod\limits_{i=1}^{m} w_{i}\sigma_{j_i}w^{-1}_{i}$.

By \cite[Theorem 3]{lp}, the $r$-fold cyclic cover of $B^4$ branched over $\check{F}$, $\Sigma_{r}(\check{F})$, is a Stein filling of $\Sigma_{r}(L)$. It is well-known that $b_{1}(\Sigma_{r}(\check{F}))\leq b_{1}(\Sigma_{r}(L))$. By Lemma~\ref{sig0}, 
$b_{1}(\Sigma_{r}(L))=\sum\limits_{k=1}^{r-1}\eta_{L}(\zeta^{k})$. So $b_{1}(\Sigma_{r}(\check{F}))\leq \sum\limits_{k=1}^{r-1}\eta_{L}(\zeta^{k})$. Moreover, $b_{3}(\Sigma_{r}(\check{F}))=0$. The Euler characteristic
$$\chi(\Sigma_{r}(\check{F}))=r\chi(B^4)-(r-1)\chi(\check{F})=r-(r-1)(n-m).$$ 
So $$b_{2}(\Sigma_{r}(\check{F}))-b_{1}(\Sigma_{r}(\check{F}))=(r-1)(1-n+m),$$
and hence
$$b_{2}(\Sigma_{r}(\check{F}))\leq \sum\limits_{k=1}^{r-1}\eta_{L}(\zeta^{k})+(r-1)(1-n+m).$$


Let $\hat{F}$ be the symplectic hat of $L$ in $\mathbb{C}P^{2}-\text{Int}(B^4)$ of degree 6 when $r=2$, in $\mathbb{C}P^{1}\times \mathbb{C}P^{1}-\text{Int}(B^4)$ of bidegree $(3,3)$ when $r=3$ or in  $\mathbb{C}P^{2}-\text{Int}(B^4)$ of degree 4 when $r=4$. Then $\check{F}\cup \hat{F}$ is a symplectic surface in $\mathbb{C}P^{2}$ of degree 6 when $r=2$, in $\mathbb{C}P^{1}\times \mathbb{C}P^{1}$ of bidegree $(3,3)$ when $r=3$ or in  $\mathbb{C}P^{2}$ of degree 4 when $r=4$.  According to \cite[Section 6.3]{eg},  the $r$-fold cyclic cover of $\mathbb{C}P^{2}$ (or of $\mathbb{C}P^{1}\times \mathbb{C}P^{1}$ when $r=3$) branched over $\check{F}\cup \hat{F}$ is the K3 surface. Therefore $C=K3-\text{Int}(\Sigma_{r}(\check{F}))$ is a Calabi-Yau cap of $\Sigma_{r}(L)$.

Consider the Mayer-Vietoris exact sequence
$$0\rightarrow H_{2}(\Sigma_{r}(L);\mathbb{Q})\rightarrow  H_{2}(\Sigma_{r}(\check{F});\mathbb{Q})\oplus H_{2}(C;\mathbb{Q}) \rightarrow H_{2}(K3;\mathbb{Q})$$
$$\rightarrow H_{1}(\Sigma_{r}(L);\mathbb{Q})\rightarrow H_{1}(\Sigma_{r}(\check{F});\mathbb{Q})\oplus H_{1}(C;\mathbb{Q})\rightarrow H_{1}(K3;\mathbb{Q})\rightarrow 0.$$

The intersection form of the K3 surface is $2E_8\oplus3H$, where $H$ is the hyperbolic quadratic form. So $H_{2}(K3;\mathbb{Q})=\mathbb{Q}^{22}$ and $H_{1}(K3;\mathbb{Q})=0$.

Immediately, we have $$b_{1}(C)\leq b_{1}(\Sigma_{r}(L))=\sum\limits_{k=1}^{r-1}\eta_{L}(\zeta^{k}).$$ By the exactness of the sequence, $$b_{2}(\Sigma_{r}(\check{F}))+b_{2}(C)-b_{2}(\Sigma_{r}(L))+b_{1}(\Sigma_{r}(L))-b_{1}(\Sigma_{r}(\check{F}))-b_{1}(C)=22.$$  Since $b_{2}(\Sigma_{r}(L))=b_{1}(\Sigma_{r}(L))$, we have $$b_{2}(C)\geq 22-b_{2}(\Sigma_{r}(\check{F}))\geq 22-(r-1)(1-n+m)-\sum\limits_{k=1}^{r-1}\eta_{L}(\zeta^{k}).$$

By \cite[P180]{lmy}, $b_{2}^{0}(C)\leq b_{1}(\Sigma_{r}(L))= \sum\limits_{k=1}^{r-1}\eta_{L}(\zeta^{k})$. So $$b^{+}_{2}(C)+b^{-}_{2}(C)=b_{2}(C)-b_{2}^{0}(C)\geq 22-(r-1)(1-n+m)-2\sum\limits_{k=1}^{r-1}\eta_{L}(\zeta^{k}).$$

On the other hand, by the Novikov Additivity Theorem \cite[Theorem 5.3]{k} and Lemma~\ref{sig1}, the signature of $C$ $$\sigma(C)=\sigma(K3)-\sigma(\Sigma_{r}(\check{F}))=-16-\sum\limits_{k=1}^{r-1}\sigma_{L}(\zeta^{k}).$$

Therefore, $$b^{+}_{2}(C)\geq \frac{1}{2}(6-(r-1)(1-n+m)-2\sum\limits_{k=1}^{r-1}\eta_{L}(\zeta^{k})-\sum\limits_{k=1}^{r-1}\sigma_{L}(\zeta^{k})),$$ and
$$b^{-}_{2}(C)\geq \frac{1}{2}(38-(r-1)(1-n+m)-2\sum\limits_{k=1}^{r-1}\eta_{L}(\zeta^{k})+\sum\limits_{k=1}^{r-1}\sigma_{L}(\zeta^{k})).$$

Let $W$ be any exact symplectic filling of $\Sigma_{r}(L)$ and $X=W\cup C$. By the assumption,
$b_{2}^{+}(X)\geq b_{2}^{+}(C)\geq 2$,
$b_{2}(X)\geq b_{2}^{-}(X)\geq b_{2}^{-}(C)\geq 7$.
According to the same argument as in the proof of \cite[Proposition 3.1]{sv} (or of \cite[Proposition 6.3]{eg}), $X$ has symplectic Kodaira dimension zero \cite{li1}, and hence has the same rational homology as that of a K3 surface, an Enriques surface or a torus bundle over torus \cite{ms, ba, li2}. Since the Enriques surface has $b_{2}^{+}=1$ and a torus bundle over torus has $b_2\leq 6$, $X$ must be a rational homology K3 surface. By a covering space argument as in  the proof of \cite[Proposition 3.1]{sv} (or of \cite[Proposition 6.3]{eg}),   $X$ is actually an integer homology K3 surface.  Consequently,  $W$ is spin, $$\chi(W)=\chi(K3)-\chi(C)=\chi(\Sigma_{r}(\check{F}))=r-(r-1)(n-m)$$ and $$\sigma(W)=\sigma(K3)-\sigma(C)=\sigma(\Sigma_{r}(\check{F}))=\sum\limits_{k=1}^{r-1}\sigma_{L}(\zeta^{k}).$$
\end{proof}

\begin{proof}[Proof of Theorem~\ref{rfoldcover}]

Since $\sum\limits_{k=1}^{r-1}\eta_{L}(\zeta^{k})=0$, $\Sigma_{r}(L)$ is a rational homology sphere.  So we have $b_{1}(\Sigma_{r}(\check{F}))=0$ and $b_{3}(\Sigma_{r}(\check{F}))=0$. The Euler characteristic $$\chi(\Sigma_{r}(\check{F}))=r-(r-1)(n-m)$$ implies that $$b_{2}(\Sigma_{r}(\check{F}))=(r-1)(1-n+m).$$

Since  $b_{2}^{0}(\Sigma_{r}(\check{F}))\leq b_{1}(\Sigma_{r}(L))$, $b_{2}^{0}(\Sigma_{r}(\check{F}))=0$. So by Lemma~\ref{sig1}, $$b_{2}^{+}(\Sigma_{r}(\check{F}))=\frac{1}{2}((r-1)(1-n+m)+\sum\limits_{k=1}^{r-1}\sigma_{L}(\zeta^{k})).$$


Let $C=K3-\text{Int}(\Sigma_{r}(\check{F}))$ be the Calabi-Yau cap of $\Sigma_{r}(L)$. Then $b_{1}(C)=0$, $$b_{2}(C)=22-b_{2}(\Sigma_{r}(\check{F}))=22-(r-1)(1-n+m),$$ and $$b_{2}^{+}(C)=3-b_{2}^{+}(\Sigma_{r}(\check{F}))=3-\frac{1}{2}((r-1)(1-n+m)+\sum\limits_{i=1}^{r-1}\sigma_{L}(\zeta^{i})).$$ By assumption, $b_{2}(C)\geq 7$ and $b_{2}^{+}(C)\geq 2$. Applying \cite[Proposition 6.3]{eg},  we conclude that $W\cup C$ is the K3 surface for any exact symplectic filling $W$ of $\Sigma_{r}(K)$. So $W$ is spin, $b_{1}(W)=0$, $\chi(W)=\chi(\Sigma_{r}(\check{F}))=r-(r-1)(n-m)$ and $\sigma(W)=\sigma(\Sigma_{r}(\check{F}))=\sum\limits_{k=1}^{r-1}\sigma_{L}(\zeta^{k})$.
\end{proof}

Now we prove Theorems~\ref{2foldcover}, \ref{3foldcover} and \ref{4foldcover}. According to \cite[Section 2.14]{go}, for $r=2,3,4$ and $K$ is a knot, $\Sigma_{r}(K)$ is a rational homology sphere, that is $\sum\limits_{k=1}^{r-1}\eta_{K}(\zeta^{k})=0$.  So we apply Theorem~\ref{rfoldcover}.

\begin{proof}[Proof of Theorem~\ref{2foldcover}]
By Lemmas~\ref{cobordism1}, \ref{cobordism2}, \ref{cobordism3} and \ref{cobordism4}, all knots considered admit a degree-6 hats.

(1) The knots $3_1$, $5_2$, $7_2$, $7_4$, $8_{21}$, $9_{2}$, $9_{5}$, $9_{35}$, $m(9_{45})$, $m(10_{126})$, $10_{131}$, $10_{133}$, $m(10_{143})$, $m(10_{148})$, $10_{159}$ and $m(10_{165})$ all have a quasi-positive braid presentation with $n-m=-1$, and have signature $-2$. So $1-n+m=2$ and $\sigma_{K}(-1)=-2$. By Theorem~\ref{rfoldcover}, any exact filling $W$ of $\Sigma_{2}(K)$ has $b_{1}(W)=0$, $\chi(W)=3$ and $\sigma(W)=-2$.

(2) The knots $5_1$,  $7_3$, $7_5$, $8_{15}$, $9_{4}$, $9_{7}$, $9_{10}$, $9_{13}$, $9_{18}$, $9_{23}$, $9_{38}$, $9_{49}$, $10_{53}$, $10_{55}$, $10_{63}$, $10_{101}$, $10_{120}$, $10_{127}$, $10_{149}$ and  $m(10_{157})$  all have a quasi-positive braid presentation with $n-m=-3$, and have signature $-4$. So $1-n+m=4$ and $\sigma_{K}(-1)=-4$. By Theorem~\ref{rfoldcover}, any exact filling $W$ of $\Sigma_{2}(K)$ has $b_{1}(W)=0$,  $\chi(W)=5$ and $\sigma(W)=-4$.


The knot $m(10_{145})$ has a quasi-positive braid presentation with $n-m=-3$, and has signature $-2$. So $1-n+m=4$ and $\sigma_{K}(-1)=-2$. By Theorem~\ref{rfoldcover}, any exact filling $W$ of $\Sigma_{2}(K)$ has $b_{1}(W)=0$,  $\chi(W)=5$ and $\sigma(W)=-2$.


(3) The knots $7_1$,  $8_{19}$, $9_3$, $9_6$, $9_{9}$, $9_{16}$, $10_{49}$,  $10_{66}$, $10_{80}$, $10_{128}$, $10_{134}$ and $10_{142}$ all have a quasi-positive presentation with $n-m=-5$, and have signature $-6$. So $1-n+m=6$ and $\sigma_{K}(-1)=-6$. By Theorem~\ref{rfoldcover}, any exact filling $W$ of $\Sigma_{2}(K)$ has $b_{1}(W)=0$,  $\chi(W)=7$ and $\sigma(W)=-6$.


Both $10_{154}$ and $10_{161}$ have a quasi-positive presentation with $n-m=-5$, and have signature $-4$. So $1-n+m=6$ and $\sigma_{K}(-1)=-4$. By Theorem~\ref{rfoldcover}, any exact filling $W$ of $\Sigma_{2}(K)$ has $b_{1}(W)=0$,  $\chi(W)=7$ and $\sigma(W)=-4$.


(4) The knot $9_{1}$ have a quasi-positive presentation with $n-m=-7$, and have signature $-8$. So $1-n+m=8$ and $\sigma_{K}(-1)=-8$. By Theorem~\ref{rfoldcover}, any exact filling $W$ of $\Sigma_{2}(K)$ has $b_{1}(W)=0$,  $\chi(W)=9$ and $\sigma(W)=-8$.


Both $10_{139}$ or $10_{152}$ have quasi-positive presentation with $n-m=-7$, and have signature $-6$. So $1-n+m=8$ and $\sigma_{K}(-1)=-6$.  By Theorem~\ref{rfoldcover}, any exact filling $W$ of $\Sigma_{2}(K)$ has $b_{1}(W)=0$,  $\chi(W)=9$ and $\sigma(W)=-6$.
\end{proof}

\begin{proof}[Proof of Theorem~\ref{3foldcover}]
(1) By Lemma~\ref{cobordism2}, the knots $3_1$, $5_{2}$, $7_{2}$, $7_{4}$, $8_{21}$, $9_{2}$, $9_{5}$, $9_{35}$, $m(9_{45})$, $m(10_{126})$, $10_{131}$, $m(10_{143})$, $m(10_{148})$, $10_{159}$ and $m(10_{165})$  all wear Hirzebruch hats with bidegree $(3,3)$. The quasi-positive representations of these knots all satisfy that $n-m=-1$. So $2(1-n+m)=4$. We compute that $\sigma_{K}(e^{\frac{2\pi i}{3}})+\sigma_{K}(e^{\frac{4\pi i}{3}})=-4$, where $\sigma_{K}$ is the signature function of $K$.  For instance, according to Knotinfo \cite{lm}, the knot $3_1$ has the signature function $$\{\{0.3333333333, \{0, -1, -2\}, 1\}\}.$$ This means that $$\sigma_{3_1}(e^{s\pi i})=\left\{
\begin{array}{cc}
 0,    &  0\leq s <\frac{1}{3}, \\
 -1,   &  s=\frac{1}{3},\\
 -2,   &  \frac{1}{3}<s\leq1. 
\end{array}\right. 
$$
Hence $\sigma_{3_1}(e^{\frac{2\pi i}{3}})+\sigma_{3_1}(e^{\frac{4\pi i}{3}})=2\sigma_{3_1}(e^{\frac{2\pi i}{3}})=2\times (-2)=-4.$ By Theorem~\ref{rfoldcover}, any exact filling $W$ of $\Sigma_{3}(K)$ has  $b_{1}(W)=0$,  $\chi(W)=5$ and $\sigma(W)=-4$. 

Using Lemmas~\ref{cobordism3} and ~\ref{cobordism4}, one can prove cases (2) and (3) in a similar way as that of case (1).

(4) By Lemma~\ref{cobordism4}, the knots $7_1$, $8_{19}$, $9_3$, $10_{128}$ and $10_{134}$  all  wear a Hirzebruch hat with bidegree $(3,3)$. The quasi-positive representations of these knots all satisfy that $n-m=-5$. So $b_{2}(\Sigma_{3}(\check{F}))=2(1-n+m)=12$ and the cap $C$ has $b_{2}(C)=10$. We compute that $\sigma_{K}(e^{\frac{2\pi i}{3}})+\sigma_{K}(e^{\frac{4\pi i}{3}})=-8$, where $\sigma_{K}$ is the signature function of $K$. So $b_{2}^{+}(\Sigma_{3}(\check{F}))=2$, and the cap $C$ has $b_{2}^{+}(C)=1$. Thus, if $W$ is not negative definite, then $W\cup C$ must be a K3 surface. Therefore, $b_{1}(W)=0$, $\chi(W)=13$ and $\sigma(W)=-8$.
\end{proof}



\begin{proof}[Proof of Theorem~\ref{4foldcover}]

(1) By Lemma~\ref{cobordism2}, the knots $3_1$, $5_2$, $7_2$, $7_4$, $8_{21}$, $m(9_{45})$,  $m(10_{126})$, $m(10_{143})$, $m(10_{148})$ and $10_{159}$ all wear projective hats of degree-4. All these knots satisfy that  $3(1-n+m)=6$, and $\sigma_{K}(i)+\sigma_{K}(-1)+\sigma_{K}(-i)=-6$. So by Theorem~\ref{rfoldcover}, any exact filling $W$ of $\Sigma_{4}(K)$ has $b_{2}(W)=6$ and $b^{+}_{2}(W)=0$.

The proof of Case (2) is similar to that of case (1). The proof of case (3) is similar to that of case (4) in the proof of Theorem~\ref{3foldcover}.
\end{proof}

Then we prove Theorems \ref{2foldcoverlink}, \ref{3foldcoverlink} and \ref{4foldcoverlink}.

\begin{proof}[Proof of Theorem~\ref{2foldcoverlink}]
By Lemma~\ref{cobordism6}, all transverse links in this theorem wear a projective hat of degree $6$.

(1) Suppose $L$ is either $m(L10n104\{1,0,0\})$ or $m(L10n104\{1,1,0\})$. Then $L$ has a quasi-positive representation with $n-m=2$, $\sigma(L)=0$ and $\eta(L)=1$. So both the inequalities in Theorem \ref{rfold} hold. Hence, for any exact filling $W$ of $\Sigma_{2}(L)$, we have $\chi(W)=0$ and $\sigma(W)=0$. Therefore, $b_{1}(W)=1$, $b_{2}(W)=0$.

(2) Suppose $L$ is  $L10n94\{1,0\}$, $m(L11n381\{0,0\})$, $m(L11n381\{0,1\})$ or $L11n428\{1,0\}$. Then $L$ has a quasi-positive representation with $n-m=1$, $\sigma(L)=-1$ and $\eta(L)=1$. So both the inequalities in Theorem \ref{rfold}  hold. Hence, for any exact filling $W$ of $\Sigma_{2}(L)$, we have $\chi(W)=1$ and $\sigma(W)=-1$. Therefore, $b_{1}(W)=1$, $b^{+}_{2}(W)=0$, $b^{-}_{2}(W)=1$ and $b^{0}_{2}(W)=0$.

(3) Suppose $L$ is  $m(L11n226\{0\})$. Then $L$ has a quasi-positive representation with $n-m=0$, $\sigma(L)=-2$ and $\eta(L)=1$. So both the inequalities in Theorem \ref{rfold}  hold. Hence, for any exact filling $W$ of $\Sigma_{2}(L)$, we have $\chi(W)=2$ and $\sigma(W)=-2$. Therefore, $b_{1}(W)=1$, $b^{+}_{2}(W)=0$, $b^{-}_{2}(W)=1$ and $b^{0}_{2}(W)=0$.

(4) Suppose $L$ is $L11n204\{1\}$. Then $L$ has a quasi-positive representation with $n-m=-2$, $\sigma(L)=0$, and $\eta(L)=1$. So both the inequalities in Theorem \ref{rfold} hold. Hence, for any exact filling $W$ of $\Sigma_{2}(L)$, we have $\chi(W)=2$ and $\sigma(W)=0$. 

The proof of the remaining cases are similar.
\end{proof}

\begin{proof}[Proof of Theorem~\ref{3foldcoverlink}]
By Lemma~\ref{cobordism6}, all transverse links in this theorem wear a Hirzebruch hat of bidegree $(3,3)$. According to the following data, all these transverse links satisfy the two inequalities in Theorem~\ref{rfold}. 

(1) Both the links  $m(L10n104\{1,0,0\})$ and $m(L10n104\{1,1,0\})$ have quasi-positive representatives with $n-m=2$, satisfy $\sigma_{L}(e^{\frac{2\pi i}{3}})+\sigma_{L}(e^{\frac{4\pi i}{3}})=0$ and $\eta_{L}(e^{\frac{2\pi i}{3}})+\eta_{L}(e^{\frac{4\pi i}{3}})=2$. 

(2) The links $m(L11n381\{0,0\})$,  $m(L11n381\{0,1\})$ and $L11n428\{1,0\}$ have quasi-positive representatives with $n-m=1$, satisfy $\sigma_{L}(e^{\frac{2\pi i}{3}})+\sigma_{L}(e^{\frac{4\pi i}{3}})=-2$ and $\eta_{L}(e^{\frac{2\pi i}{3}})+\eta_{L}(e^{\frac{4\pi i}{3}})=2$. 

(3) The link $m(L10n94\{1,0\})$ has a quasi-positive representative with $n-m=1$, satisfies $\sigma_{L}(e^{\frac{2\pi i}{3}})+\sigma_{L}(e^{\frac{4\pi i}{3}})=0$ and $\eta_{L}(e^{\frac{2\pi i}{3}})+\eta_{L}(e^{\frac{4\pi i}{3}})=0$. 

(4) The link $m(L11n226\{0\})$ has a quasi-positive representative with $n-m=0$, satisfies $\sigma_{L}(e^{\frac{2\pi i}{3}})+\sigma_{L}(e^{\frac{4\pi i}{3}})=-2$ and $\eta_{L}(e^{\frac{2\pi i}{3}})+\eta_{L}(e^{\frac{4\pi i}{3}})=0$. 

The computation of the Levine-Tristram signatures and nullities are based on the Seifert matrix provided in Knotinfo  \cite{lm}.  Note that the given Seifert matrices and classical nullities in Knotinfo \cite{lm} imply that the choosing Seifert surfaces of these links are all connected.

In all these cases, the Betti numbers of the exact fillings are uniquely determined by the Euler characteristic and signatures, and are obtained as stated. 
\end{proof}

\begin{proof}[Proof of Theorem~\ref{4foldcoverlink}]
By Lemma~\ref{cobordism6}, all transverse links in this theorem wear a projective hat of degree $4$. According to the following data, all these transverse links satisfy the two inequalities in Theorem~\ref{rfold}. 

(1) Both the links  $m(L10n104\{1,0,0\})$ and $m(L10n104\{1,1,0\})$ have quasi-positive representatives with $n-m=2$, satisfy  $\sigma_{L}(i)+\sigma_{L}(-1)+\sigma_{L}(-i)=0$ and $\eta_{L}(i)+\eta_{L}(-1)+\eta_{L}(-i)=3$. 

(2) The links $m(L11n381\{0,0\})$,  $m(L11n381\{0,1\})$ and $L11n428\{1,0\}$  have quasi-positive representatives with $n-m=1$, satisfy $\sigma_{L}(i)+\sigma_{L}(-1)+\sigma_{L}(-i)=-3$ and $\eta_{L}(i)+\eta_{L}(-1)+\eta_{L}(-i)=3$. 

(3) The link $m(L10n94\{1,0\})$ has a quasi-positive representative with $n-m=1$, satisfies $\sigma_{L}(i)+\sigma_{L}(-1)+\sigma_{L}(-i)=-1$ and $\eta_{L}(i)+\eta_{L}(-1)+\eta_{L}(-i)=0$. 

(4) The link $m(L11n226\{0\})$ has a quasi-positive representative with $n-m=0$, satisfies $\sigma_{L}(i)+\sigma_{L}(-1)+\sigma_{L}(-i)=-4$ and $\eta_{L}(i)+\eta_{L}(-1)+\eta_{L}(-i)=1$. 

In all these cases, the Betti numbers of the exact fillings are uniquely determined by the Euler characteristic and signatures, and are obtained as stated.
\end{proof}

\section{Symplectic hats}
According to the main theorem in \cite{flmn}, the torus knots $T_{5,6}$ and $T_{3,11}$ wear projective hats of degree $6$, $T_{3,4}$ and $T_{2,7}$ wear projective hats of degree $4$. By \cite[Proposition 4.10]{eg}, $T_{3,5}$ wears a Hirzebruch hat of bidegree $(3,3)$. The main task of this section is to show that many quasi-positive links are relative symplectic cobordant to these torus knots, and hence wear corresponding symplectic hats. In \cite{eg}, Etnyre and Golla introduced some efficient methods to construct relative symplectic cobordism. Given a transverse link $L_{-}$ with a closed braid representation, then \cite[Lemma 2.8]{eg} tells us that there is a relative symplectic cobordism from $L_{-}$ to $L_{+}$ which is obtained by adding a positive generator to the closed braid representing $L_{-}$. 

\begin{lemma}\label{cobordism1}
The torus link $T_{2,8}$ in $(S^{3}, \xi_{st})$ has a relative symplectic cobordism to the torus knot $T_{3,5}$ in a piece of the symplectization of $(S^{3}, \xi_{std})$. In particular, the torus link $T_{2,q}$ wears a Hirzebruch hat of bidegree $(3,3)$ if $3\leq q\leq 8$.
\end{lemma}

\begin{proof}
In the following, we denote with $\sigma _{i}$ the generators of the braid group. We use $\nearrow_{k}$ to denote the insertion of $k$ positive generators, $\sim $ to denote conjugation, $\sim_{D}$ to denote Markov destabilisation, and $\sim_{S}$ to denote Markov stabilization.  We construct a relative symplectic cobordism from $T_{2,8}$ to $T_{3,5}$ as follows: 

$T_{2,8}=\sigma_{1}^{8}\sim_{S}\sigma_{1}^{2}\sigma_{2}\sigma_{1}^{6}=(\sigma_{1}\sigma_{2})^{2}\sigma_{1}^{5}\nearrow_{1}(\sigma_{1}\sigma_{2})^{2}\sigma_{1}^{2}\sigma_{2}\sigma_{1}^{3}=(\sigma_{1}\sigma_{2})^{4}\sigma_{1}^{2}=(\sigma_{2}\sigma_{1})^{5}=T_{3,5}$,

where both the first and the last $``="$ imply that the links are the closures of the braids. Obviously, for any $3\leq q<8$, $T_{2,q}$ is relative symplectic cobordant to $T_{2,8}$. Since $T_{3,5}$ wears a Hirzebruch hat of bidegree $(3,3)$, so is $T_{2,q}$ for $3\leq q\leq8$.
\end{proof}

\begin{lemma}\label{cobordism2}
The quasi-positive representative of the knot $K$ in the knot table wears 

(1) a projective hat of degree-$6$ if $K$ is either $3_1$, $5_2$, $7_2$, $7_4$, $8_{21}$, $9_{2}$, $9_{5}$, $9_{35}$, $m(9_{45})$, $m(10_{126})$, $10_{131}$, $10_{133}$, $m(10_{143})$, $m(10_{148})$, $10_{159}$ or $m(10_{165})$;

(2) a Hirzebruch hat of bidegree-$(3,3)$ if $K$ is either $3_1$, $5_{2}$, $7_{2}$, $7_{4}$, $8_{21}$, $9_{2}$, $9_{5}$, $9_{35}$, $m(9_{45})$, $m(10_{126})$, $10_{131}$,
$m(10_{143})$, $m(10_{148})$, $10_{159}$ or $m(10_{165})$;

(3) a projective hat of degree-$4$ if $K$ is either $3_1$, $5_2$, $7_2$, $7_4$, $8_{21}$, $m(9_{45})$,  $m(10_{126})$, $m(10_{143})$, $m(10_{148})$ or $10_{159}$.

All these transverse knots have self-linking numbers  $1$.
\end{lemma}

\begin{proof} Below we adopt the notations in the proof of Lemma~\ref{cobordism1}. We construct relative symplectic cobordisms as follows. 

$3_{1}=\sigma _{1}^{3}=T_{2,3}$.


$5_{2}=\sigma_{1}^{2}\sigma_{2}\cdot\sigma_{2}\sigma_{1}\sigma_{2}^{-1}\nearrow_{1}\sigma_{1}^{2}\sigma _{2}^{2}\sigma_{1}\sim \sigma_{1}^{3}\sigma_{2}^{2}\nearrow_{1}\sigma_{1}^{3}\sigma_{2}\sigma_{1}\sigma_{2}=\sigma _{1}^{3}\sigma_{1}\sigma_{2}\sigma_{1}\sim_{D}\sigma_{1}^{5}=T_{2,5}$.

$7_{2}=\sigma^{2}_{1}\cdot\sigma_{3}\sigma_{2}\sigma^{-1}_{3}\cdot\sigma_{2}\sigma_{1}\sigma^{-1}_{2}\cdot\sigma_{3}\nearrow_{2}\sigma^{2}_{1}\sigma_{3}\sigma_{2}^{2}\sigma_{1}\sigma_{3}\sim \sigma _{1}^{3}\sigma _{2}^{2}\sigma _{3}^{2}\nearrow_{1}\sigma _{1}^{3}\sigma _{2}^{2}\sigma _{3}\sigma _{2}\sigma _{3}=\sigma _{1}^{3}\sigma _{2}^{2}\sigma _{2}\sigma _{3}\sigma _{2}$
$\sim _{D}\sigma _{1}^{3}\sigma _{2}^{4}\nearrow_{1}\sigma _{1}^{3}\sigma _{2}\sigma _{1}\sigma _{2}^{3}=\sigma _{1}^{3}\sigma _{1}^{3}\sigma _{2}\sigma _{1}\sim _{D}\sigma _{1}^{7}=T_{2,7}$.

$7_{4}=\sigma_{1}\cdot\sigma_{3}\sigma_{2}\sigma^{-1}_{3}\cdot\sigma_{3}\sigma_{2}\sigma_{1}\sigma^{-1}_{2}\sigma^{-1}_{3}\cdot\sigma_{2}\sigma_{1}\sigma^{-1}_{2}\cdot\sigma_{3}\nearrow_{2}\sigma_{1}\sigma_{3}\sigma_{2}^{2}\sigma_{1}^{2}\sigma_{3}\sim\sigma_{1}^{3}\sigma_{2}^{2}\sigma_{3}^{2}\nearrow_{1}\sigma_{1}^{3}\sigma_{2}^{2}\sigma_{3}\sigma_{2}\sigma_{3}$\\
$=\sigma_{1}^{3}\sigma_{2}^{2}\sigma_{2}\sigma_{3}\sigma_{2}\sim_{D}\sigma_{1}^{3}\sigma_{2}^{4}\nearrow_{1}\sigma_{1}^{3}\sigma_{2}\sigma_{1}\sigma_{2}^{3}=\sigma_{1}^{3}\sigma_{1}^{3}\sigma_{2}\sigma_{1}\sim_{D}\sigma_{1}^{7}=T_{2,7}$.

$8_{21}=\sigma_{2}\sigma_{1}^{2}\sigma^{-1}_{2}\cdot\sigma^{-1}_{2}\sigma_{1}\sigma_{2}\cdot\sigma_{2}\nearrow_{2}\sigma_{2}\sigma_{1}^{3}\sigma_{2}^{2}\sim\sigma_{1}^{3}\sigma_{2}^{3}\nearrow_{1}\sigma_{1}^{3}\sigma_{2}\sigma_{1}\sigma_{2}^{2}=\sigma_{1}^{3}\sigma_{1}^{2}\sigma_{2}\sigma_{1}\sim_{D}\sigma_{1}^{6}$\\
$=T_{2,6}$.

$9_{2}=\sigma_{2}\cdot\sigma_{2}\sigma_{1}\sigma_{2}^{-1}\cdot\sigma_{4}\sigma_{3}\sigma_{2}\sigma_{3}^{-1}\sigma_{4}^{-1}\cdot\sigma_{2}\cdot\sigma_{3}\sigma_{2}\sigma_{1}\sigma_{2}^{-1}\sigma_{3}^{-1}\cdot\sigma_{4}\nearrow_{4}\sigma_{2}^{2}\sigma_{1}\sigma_{4}\sigma_{3}\sigma_{2}^{2}\sigma_{3}\sigma_{2}\sigma_{1}\sigma_{2}^{-1}\sigma_{4}$\\
$\sim\sigma_{4}\sigma_{3}\sigma_{2}^{2}\sigma_{3}\sigma_{2}\sigma_{1}\sigma_{2}^{-1}\sigma_{4}\sigma_{2}^{2}\sigma_{1}=\sigma_{4}\sigma_{3}\sigma_{2}^{2}\sigma_{3}\sigma_{2}\sigma_{1}\sigma_{2}\sigma_{1}\sigma_{4}=\sigma_{4}\sigma_{3}\sigma_{2}^{2}\sigma_{3}\sigma_{2}^{2}\sigma_{1}\sigma_{2}\sigma_{4}\sim_{D}\sigma_{4}\sigma_{3}\sigma_{2}^{2}\sigma_{3}\sigma_{2}^{3}\sigma_{4}$\\
$\sim\sigma_{4}^{2}\sigma_{3}\sigma_{2}^{2}\sigma_{3}\sigma_{2}^{3}\nearrow_{1}\sigma_{4}\sigma_{3}\sigma_{4}\sigma_{3}\sigma_{2}^{2}\sigma_{3}\sigma_{2}^{3}=\sigma_{3}\sigma_{4}\sigma_{3}^{2}\sigma_{2}^{2}\sigma_{3}\sigma_{2}^{3}\sim_{D}\sigma_{3}^{3}\sigma_{2}^{2}\sigma_{3}\sigma_{2}^{3}=\sigma_{3}^{2}(\sigma_{3}\sigma_{2})^{3}\sigma_{2}$\\
$\nearrow_{1}\sigma_{3}\sigma_{2}\sigma_{3}(\sigma_{3}\sigma_{2})^{3}\sigma_{2}=(\sigma_{3}\sigma_{2})^{5}=T_{3,5}$.

$9_{5}=\sigma_{1}\sigma_{2}\cdot\sigma_{2}\sigma_{1}\sigma_{2}^{-1}\cdot\sigma_{4}\sigma_{3}\sigma_{4}^{-1}\cdot\sigma_{3}\sigma_{2}\sigma_{1}\sigma_{2}^{-1}\sigma_{3}^{-1}\cdot\sigma_{4}\nearrow_{4}\sigma_{1}\sigma_{2}^{2}\sigma_{1}\sigma_{4}\sigma_{3}^{2}\sigma_{2}\sigma_{1}\sigma_{4}\sim\sigma_{2}^{2}\sigma_{1}\sigma_{4}\sigma_{3}^{2}\sigma_{2}\sigma_{1}\sigma_{4}\sigma_{1}$\\
$=\sigma_{2}^{2}\sigma_{4}\sigma_{3}^{2}\sigma_{1}\sigma_{2}\sigma_{1}^{2}\sigma_{4}=\sigma_{2}^{2}\sigma_{4}\sigma_{3}^{2}\sigma_{2}^{2}\sigma_{1}\sigma_{2}\sigma_{4}\sim_{D}\sigma_{2}^{2}\sigma_{4}\sigma_{3}^{2}\sigma_{2}^{3}\sigma_{4}\sim\sigma_{2}^{5}\sigma_{3}^{2}\sigma_{4}^{2}\nearrow_{1}\sigma_{2}^{5}\sigma_{3}^{2}\sigma_{4}\sigma_{3}\sigma_{4}$\\
$=\sigma_{2}^{5}\sigma_{3}^{3}\sigma_{4}\sigma_{3}\sim_{D}\sigma_{2}^{5}\sigma_{3}^{4}\nearrow_{1}\sigma_{2}^{3}\sigma_{3}\sigma_{2}^{2}\sigma_{3}^{4}=\sigma_{2}(\sigma_{2}\sigma_{3})^{3}\sigma_{3}^{3}=(\sigma_{2}\sigma_{3})^{4}\sigma_{3}^{2}\sim\sigma_{3}(\sigma_{2}\sigma_{3})^{4}\sigma_{3}$\\
$=(\sigma_{2}\sigma_{3})^{5}=T_{3,5}$.


$9_{35}=\sigma_{1}\sigma_{2}\sigma_{3}\sigma_{4}\cdot\sigma_{4}\sigma_{3}\sigma_{2}\sigma_{3}^{-1}\sigma_{4}^{-1}\cdot\sigma_{3}\sigma_{2}\sigma_{1}\sigma_{2}^{-1}\sigma_{3}^{-1}\nearrow_{3}\sigma_{1}\sigma_{2}\sigma_{3}\sigma_{4}^{2}\sigma_{3}\sigma_{2}^{2}\sigma_{1}\sim\sigma_{1}^{2}\sigma_{2}\sigma_{3}\sigma_{4}^{2}\sigma_{3}\sigma_{2}^{2}$\\
$\nearrow_{2}\sigma_{1}\sigma_{2}\sigma_{1}\sigma_{2}\sigma_{3}\sigma_{4}\sigma_{3}\sigma_{4}\sigma_{3}\sigma_{2}^{2}=\sigma_{2}\sigma_{1}\sigma_{2}^{2}\sigma_{3}^{2}\sigma_{4}\sigma_{3}^{2}\sigma_{2}^{2}\sim_{D}\sigma_{2}^{3}\sigma_{3}^{4}\sigma_{2}^{2}\sim\sigma_{3}\sigma_{2}^{5}\sigma_{3}^{3}\nearrow_{1}\sigma_{3}\sigma_{2}^{3}\sigma_{3}\sigma_{2}^{2}\sigma_{3}^{3}$\\
$=\sigma_{3}\sigma_{2}(\sigma_{2}\sigma_{3})^{3}\sigma_{3}^{2}=\sigma_{3}(\sigma_{2}\sigma_{3})^{4}\sigma_{3}=(\sigma_{2}\sigma_{3})^{5}=T_{3,5}$.

$m(9_{45})=\sigma_{2}^{-1}\sigma_{1}^{2}\sigma_{2}\cdot\sigma_{2}\sigma_{1}\sigma_{2}^{-1}\cdot\sigma_{3}\cdot\sigma_{3}\sigma_{2}\sigma_{3}^{-1}\nearrow_{3}\sigma_{1}^{2}\sigma_{2}^{2}\sigma_{1}\sigma_{3}^{2}\sigma_{2}\sim\sigma_{2}^{2}\sigma_{3}^{2}\sigma_{1}\sigma_{2}\sigma_{1}^{2}=\sigma_{2}^{2}\sigma_{3}^{2}\sigma_{2}^{2}\sigma_{1}\sigma_{2}$\\
$\sim_{D}\sigma_{2}^{2}\sigma_{3}^{2}\sigma_{2}^{3}\nearrow_{1}\sigma_{2}^{2}\sigma_{3}\sigma_{2}\sigma_{3}\sigma_{2}^{3}=\sigma_{2}^{2}\sigma_{2}\sigma_{3}\sigma_{2}\sigma_{2}^{3}\sim_{D}\sigma_{2}^{7}=T_{2,7}$.

$m(10_{126})=\sigma^{2}_{1}\cdot\sigma^{3}_{1}\sigma_{2}\sigma^{-3}_{1}\cdot\sigma_{2}\nearrow_{4}\sigma^{5}_{1}\sigma _{2}\sigma _{1}\sigma _{2}=\sigma^{5}_{1}\sigma_{1}\sigma _{2}\sigma _{1}\sim _{D}\sigma _{1}^{7}=T_{2,7}$.

$10_{131}=\sigma_{1}\cdot\sigma_{1}^{2}\sigma_{2}^{-1}\sigma_{3}^{-1}\sigma_{2}\sigma_{4}\sigma_{2}^{-1}\sigma_{3}\sigma_{2}\sigma_{1}^{-2}\cdot\sigma_{2}^{-1}\sigma_{3}\sigma_{2}\cdot\sigma_{4}\sigma_{3}\cdot\sigma_{2}^{-1}\sigma_{3}^{-1}\sigma_{2}\sigma_{4}\sigma_{2}^{-1}\sigma_{3}\sigma_{2}$\\
$\nearrow_{5}\sigma_{1}^{3}\sigma_{4}\sigma_{3}^{2}\sigma_{2}\sigma_{4}^{2}\sigma_{3}\sigma_{2}\sim\sigma_{1}^{3}\sigma_{3}^{2}\sigma_{2}\sigma_{4}^{2}\sigma_{3}\sigma_{4}\sigma_{2}=\sigma_{1}^{3}\sigma_{3}^{2}\sigma_{2}\sigma_{3}\sigma_{4}\sigma_{3}^{2}\sigma_{2}\sim_{D}\sigma_{1}^{3}\sigma_{3}^{2}\sigma_{2}\sigma_{3}^{3}\sigma_{2}$\\
$\nearrow_{1}\sigma_{3}^{2}\sigma_{1}\sigma_{2}\sigma_{1}^{2}\sigma_{2}\sigma_{3}^{3}\sigma_{2}=\sigma_{3}^{2}\sigma_{2}^{2}\sigma_{1}\sigma_{2}^{2}\sigma_{3}^{3}\sigma_{2}\sim_{D}\sigma_{3}^{2}\sigma_{2}^{4}\sigma_{3}^{3}\sigma_{2}\sim\sigma_{2}^{4}\sigma_{3}^{3}\sigma_{2}\sigma_{3}^{2}=\sigma_{2}^{4}\sigma_{3}^{2}\sigma_{2}\sigma_{3}\sigma_{2}\sigma_{3}$\\
$\sim\sigma_{2}^{3}\sigma_{3}^{2}\sigma_{2}\sigma_{3}\sigma_{2}\sigma_{3}\sigma_{2}=\sigma_{2}^{3}\sigma_{3}\sigma_{2}\sigma_{3}\sigma_{2}\sigma_{3}\sigma_{2}\sigma_{3}\sim\sigma_{2}^{2}\sigma_{3}\sigma_{2}\sigma_{3}\sigma_{2}\sigma_{3}\sigma_{2}\sigma_{3}\sigma_{2}=(\sigma_{2}\sigma_{3})^{5}=T_{3,5}$.

$10_{133}=\sigma_{1}^{2}\sigma_{2}\cdot\sigma_{2}\sigma_{4}^{-1}\sigma_{3}^{-1}\sigma_{2}\sigma_{3}\sigma_{4}\sigma_{2}^{-1}\cdot\sigma_{2}\sigma_{3}^{-1}\sigma_{2}^{-1}\sigma_{3}\sigma_{1}^{-1}\sigma_{3}^{-1}\sigma_{2}\sigma_{3}\sigma_{4}\sigma_{3}^{-1}\sigma_{2}^{-1}\sigma_{3}\sigma_{1}\sigma_{3}^{-1}\sigma_{2}\sigma_{3}\sigma_{2}^{-1}\cdot\sigma_{2}\sigma_{3}^{-1}\sigma_{2}^{-1}\sigma_{3}\sigma_{1}\sigma_{3}^{-1}\sigma_{2}\sigma_{3}\sigma_{2}^{-1}\nearrow_{7}\sigma_{1}^{2}\sigma_{2}^{3}\sigma_{3}\sigma_{4}\sigma_{3}\sigma_{4}\sigma_{1}^{2}\sigma_{2}\sigma_{3}=\sigma_{1}^{2}\sigma_{2}^{3}\sigma_{3}^{2}\sigma_{4}\sigma_{3}\sigma_{1}^{2}\sigma_{2}\sigma_{3}$\\
$\sim_{D}\sigma_{1}^{2}\sigma_{2}^{3}\sigma_{3}^{3}\sigma_{1}^{2}\sigma_{2}\sigma_{3}=\sigma_{1}^{2}\sigma_{2}^{3}\sigma_{1}^{2}\sigma_{3}^{3}\sigma_{2}\sigma_{3}=\sigma_{1}^{2}\sigma_{2}^{3}\sigma_{1}^{2}\sigma_{2}\sigma_{3}\sigma_{2}^{3}\sim_{D}\sigma_{1}^{2}\sigma_{2}^{3}\sigma_{1}^{2}\sigma_{2}^{4}$\\
$\nearrow_{5}(\sigma_{1}\sigma_{2})^{2}\sigma_{2}(\sigma_{1}\sigma_{2})^{4}\sigma_{2}\sigma_{1}\sigma_{2}=(\sigma_{1}\sigma_{2})^{8}=T_{3,8}$.

$m(10_{143})=\sigma_{1}\cdot\sigma^{3}_{1}\sigma_{2}\sigma^{-3}_{1}\cdot\sigma^{2}_{2}\nearrow_{4}\sigma^{4}_{1}\sigma_{2}\sigma_{1}\sigma^{2}_{2}=\sigma^{4}_{1}\sigma_{1}^{2}\sigma_{2}\sigma_{1}\sim_{D}\sigma_{1}^{7}=T_{2,7}$.

$m(10_{148})=\sigma_{1}\cdot\sigma^{-1}_{2}\sigma^{-2}_{1}\sigma_{2}\sigma^{2}_{1}\sigma_{2}\cdot\sigma_{2}\cdot\sigma_{2}\sigma_{1}\sigma^{-1}_{2}\nearrow_{3}\sigma_{2}\sigma^{2}_{1}\sigma^{3}_{2}\sigma_{1}\sim\sigma_{1}\sigma_{2}\sigma^{2}_{1}\sigma^{3}_{2}=\sigma_{2}^{2}\sigma_{1}\sigma_{2}\sigma^{3}_{2}$\\
$\sim_{D}\sigma_{2}^{6}=T_{2,6}$.

$10_{159}=\sigma_{2}^{-1}\sigma_{1}^{-2}\sigma_{2}\sigma_{1}^{2}\sigma_{2}\cdot\sigma_{2}^{-1}\sigma_{1}\sigma_{2}\cdot\sigma_{2}\sigma_{1}^{2}\sigma_{2}^{-1}\nearrow_{3}\sigma_{1}^{3}\sigma_{2}^{2}\sigma_{1}^{2}\sim\sigma_{1}^{5}\sigma_{2}^{2}\nearrow_{1}\sigma_{1}^{5}\sigma_{2}\sigma_{1}\sigma_{2}$\\
$=\sigma_{1}^{5}\sigma_{1}\sigma_{2}\sigma_{1}\sim_{D}\sigma_{1}^{7}=T_{2,7}$.

$m(10_{165})=\sigma_{3}\sigma_{2}^{-1}\sigma_{1}\sigma_{2}\sigma_{3}^{-1}\cdot\sigma_{2}\cdot\sigma_{2}\sigma_{1}\sigma_{2}^{-1}\cdot\sigma_{3}\cdot\sigma_{3}\sigma_{2}\sigma_{1}\sigma_{2}^{-1}\sigma_{3}^{-1}\nearrow_{4}\sigma_{3}\sigma_{1}\sigma_{2}^{3}\sigma_{1}\sigma_{3}^{2}\sigma_{2}\sigma_{1}\sigma_{3}^{-1}$\\
$\sim\sigma_{2}^{3}\sigma_{1}\sigma_{3}^{2}\sigma_{2}\sigma_{1}^{2}=\sigma_{2}^{3}\sigma_{3}^{2}\sigma_{1}\sigma_{2}\sigma_{1}^{2}=\sigma_{2}^{3}\sigma_{3}^{2}\sigma_{2}^{2}\sigma_{1}\sigma_{2}\sim_{D}\sigma_{2}^{3}\sigma_{3}^{2}\sigma_{2}^{3}\nearrow_{1}\sigma_{2}^{3}\sigma_{3}\sigma_{2}\sigma_{3}\sigma_{2}^{3}=\sigma_{2}^{4}\sigma_{3}\sigma_{2}^{4}$\\
$\sim_{D}\sigma_{2}^{8}=T_{2,8}$.

(1) All these transverse knots are relative symplectic cobordant to $T_{3,11}$. So they all wear projective hats of degree-6. 

(2) By Lemma~\ref{cobordism1}, the knots $3_1$, $5_{2}$, $7_{2}$, $7_{4}$, $8_{21}$, $9_{2}$, $9_{5}$, $9_{35}$, $m(9_{45})$, $m(10_{126})$, $10_{131}$,
$m(10_{143})$, $m(10_{148})$, $10_{159}$ and $m(10_{165})$ are all relative symplectic cobordant to $T_{3,5}$. So they wear a Hirzebruch hat of bidegree-$(3,3)$. 

(3) The knots $3_1$, $5_2$, $7_2$, $7_4$, $8_{21}$, $m(9_{45})$, $m(10_{126})$, $m(10_{143})$, $m(10_{148})$ and $10_{159}$ are all relative symplectic cobordant to  $T_{2,7}$. So they all wear projective hats of degree-$4$.
\end{proof}

\begin{lemma}\label{cobordism3}
The quasi-positive representative of the knot $K$ in the knot table wears 

(1) a projective hat of degree-$6$ if $K$ is either $5_1$,  $7_3$, $7_5$, $8_{15}$, $9_{4}$, $9_{7}$, $9_{10}$, $9_{13}$, $9_{18}$, $9_{23}$, $9_{38}$, $9_{49}$, $10_{53}$, $10_{55}$, $10_{63}$, $10_{101}$, $10_{120}$, $10_{127}$, $10_{149}$, $m(10_{157})$ or $m(10_{145})$;

(2) a Hirzebruch hat of bidegree-$(3,3)$ if $K$ is either $5_1$,  $7_3$, $7_5$, $8_{15}$, $9_{4}$, $9_{7}$, $9_{10}$, $9_{13}$, $9_{18}$, $9_{23}$, $9_{38}$, $9_{49}$, $10_{120}$, $10_{127}$, $10_{149}$ or $m(10_{157})$;

(3) a projective hat of degree-$4$ if $K$ is either $5_1$,  $7_3$, $7_5$ or $8_{15}$.

All these transverse knots have self-linking numbers  $3$.   
\end{lemma}
\begin{proof} We construct relative symplectic cobordisms as follows. The lemma follows by the same argument as that in the proof of Lemma~\ref{cobordism2}.

$5_{1}=\sigma _{1}^{5}=T_{2,5}$.

$7_{3}=\sigma^{4}_{1}\sigma_{2}\cdot\sigma_{2}\sigma_{1}\sigma_{2}^{-1}\nearrow_{1}\sigma^{4}_{1}\sigma_{2}^{2}\sigma_{1}\sim \sigma^{5}_{1}\sigma_{2}^{2}\nearrow_{1}\sigma^{5}_{1}\sigma_{2}\sigma_{1}\sigma_{2}=\sigma^{5}_{1}\sigma_{1}\sigma_{2}\sigma_{1}\sim_{D}\sigma_{1}^{7}=T_{2,7}$.

$7_{5}=\sigma^{3}_{1}\sigma_{2}\cdot\sigma_{2}\sigma_{1}^{2}\sigma^{-1}_{2}\nearrow_{2}\sigma_{1}^{3}\sigma_{2}^{2}\sigma_{1}^{2}\sigma_{2}\sim\sigma_{1}\sigma^{2}_{2}\sigma^{2}_{1}\sigma_2\sigma_{1}^{2}=\sigma_{1}\sigma^{2}_{2}(\sigma_{1}\sigma_{2})^{2}\sigma_{1}=(\sigma_{1}\sigma_{2})^4=T_{3,4}$.


$8_{15}=\sigma_{1}\cdot\sigma_{3}\sigma_{2}^{2}\sigma_{3}^{-1}\cdot\sigma_{3}\sigma_{2}\sigma_{1}^{2}\sigma^{-1}_{2}\sigma^{-1}_{3}\cdot\sigma_{2}\sigma_{3}\nearrow_{1}\sigma_{1}\sigma_{3}\sigma_{2}^{3}\sigma_{1}^{2}\sigma_{3}\sim\sigma_{1}^{2}\sigma_{2}^{3}\sigma_{1}\sigma_{3}^{2}\nearrow_{1}\sigma_{1}^{2}\sigma_{2}^{3}\sigma_{1}\sigma_{3}\sigma_{2}\sigma_{3}$\\
$=\sigma_{1}^{2}\sigma_{2}^{3}\sigma_{1}\sigma_{2}\sigma_{3}\sigma_{2}\sim_{D}\sigma_{1}^{2}\sigma_{2}^{3}\sigma_{1}\sigma_{2}^{2}=\sigma_{1}^{2}\sigma_{2}^{2}(\sigma_{1}\sigma_{2})^{2}\sim\sigma_{1}\sigma_{2}^{2}(\sigma_{1}\sigma_{2})^{2}\sigma_{1}=(\sigma_{1}\sigma_{2})^{4}=T_{3,4}$.


$9_{4}=\sigma^{4}_{1}\cdot\sigma_{3}\sigma_{2}\sigma^{-1}_{3}\cdot\sigma_{2}\sigma_{1}\sigma^{-1}_{2}\cdot\sigma_{3}\nearrow_{2}\sigma^{4}_{1}\sigma_{3}\sigma_{2}^{2}\sigma_{1}\sigma_{3}\sim\sigma^{5}_{1}\sigma_{2}^{2}\sigma_{3}^{2}\nearrow_{1}\sigma^{5}_{1}\sigma_{2}^{2}\sigma_{3}\sigma_{2}\sigma_{3}=\sigma^{5}_{1}\sigma_{2}^{3}\sigma_{3}\sigma_{2}$\\
$\sim_{D}\sigma^{5}_{1}\sigma_{2}^{4}\sim\sigma_{2}\sigma_{1}^{5}\sigma_{2}^{3}\nearrow_{1}\sigma_{2}\sigma_{1}^{3}\sigma_{2}\sigma_{1}^{2}\sigma_{2}^{3}=\sigma_{2}\sigma_{1}(\sigma_{1}\sigma_{2})^{3}\sigma_{2}^{2}=\sigma_{2}(\sigma_{1}\sigma_{2})^{4}\sigma_{2}=(\sigma_{1}\sigma_{2})^{5}=T_{3,5}$.


$9_{7}=\sigma^{3}_{1}\cdot\sigma_{3}\sigma_{2}\sigma^{-1}_{3}\cdot\sigma_{2}\sigma_{1}\sigma^{-1}_{2}\cdot\sigma^{2}_{3}\nearrow_{2}\sigma^{3}_{1}\sigma_{3}\sigma_{2}^{2}\sigma_{1}\sigma^{2}_{3}\sim\sigma_{1}^{4}\sigma_{2}^{2}\sigma_{3}^{3}\nearrow_{1}\sigma_{1}^{4}\sigma_{2}^{2}\sigma_{3}\sigma_{2}\sigma_{3}^{2}=\sigma_{1}^{4}\sigma_{2}^{4}\sigma_{3}\sigma_{2}$\\
$\sim_{D}\sigma_{1}^{4}\sigma_{2}^{5}\sim\sigma_{1}^{3}\sigma_{2}^{5}\sigma_{1}\nearrow_{1}\sigma_{1}^{3}\sigma_{2}^{2}\sigma_{1}\sigma_{2}^{3}\sigma_{1}=\sigma_{1}^{2}(\sigma_{1}\sigma_{2})^{3}\sigma_{2}\sigma_{1}=\sigma_{1}(\sigma_{1}\sigma_{2})^{4}\sigma_{1}=(\sigma_{1}\sigma_{2})^{5}=T_{3,5}$.


$9_{10}=\sigma_{1}\cdot\sigma_{3}\sigma_{2}\sigma_{3}^{-1}\cdot\sigma_{3}\sigma_{2}\sigma_{1}^{3}\sigma_{2}^{-1}\sigma_{3}^{-1}\cdot\sigma_{2}\sigma_{1}\sigma_{2}^{-1}\cdot\sigma_{3}\nearrow_{2}\sigma_{1}\sigma_{3}\sigma_{2}^{2}\sigma_{1}^{4}\sigma_{3}\sim\sigma^{5}_{1}\sigma_{2}^{2}\sigma_{3}^{2}$\\
$\nearrow_{1}\sigma^{5}_{1}\sigma_{2}^{2}\sigma_{3}\sigma_{2}\sigma_{3}=\sigma^{5}_{1}\sigma_{2}^{3}\sigma_{3}\sigma_{2}\sim_{D}\sigma^{5}_{1}\sigma_{2}^{4}\sim\sigma_{2}\sigma_{1}^{5}\sigma_{2}^{3}\nearrow_{1}\sigma_{2}\sigma_{1}^{3}\sigma_{2}\sigma_{1}^{2}\sigma_{2}^{3}=\sigma_{2}\sigma_{1}(\sigma_{1}\sigma_{2})^{3}\sigma_{2}^{2}$\\
$=\sigma_{2}(\sigma_{1}\sigma_{2})^{4}\sigma_{2}=(\sigma_{1}\sigma_{2})^{5}=T_{3,5}$.


$9_{13}=\sigma_{1}^{3}\cdot\sigma_{3}\sigma_{2}\sigma_{3}^{-1}\cdot\sigma_{3}\sigma_{2}\sigma_{1}\sigma_{2}^{-1}\sigma_{3}^{-1}\cdot\sigma_{2}\sigma_{1}\sigma_{2}^{-1}\cdot\sigma_{3}\nearrow_{2}\sigma_{1}^{3}\sigma_{3}\sigma_{2}^{2}\sigma_{1}^{2}\sigma_{3}\sim\sigma^{5}_{1}\sigma_{2}^{2}\sigma_{3}^{2}$\\
$\nearrow_{1}\sigma^{5}_{1}\sigma_{2}^{2}\sigma_{3}\sigma_{2}\sigma_{3}=\sigma^{5}_{1}\sigma_{2}^{3}\sigma_{3}\sigma_{2}\sim_{D}\sigma^{5}_{1}\sigma_{2}^{4}\sim\sigma_{2}\sigma_{1}^{5}\sigma_{2}^{3}\nearrow_{1}\sigma_{2}\sigma_{1}^{3}\sigma_{2}\sigma_{1}^{2}\sigma_{2}^{3}=\sigma_{2}\sigma_{1}(\sigma_{1}\sigma_{2})^{3}\sigma_{2}^{2}$\\
$=\sigma_{2}(\sigma_{1}\sigma_{2})^{4}\sigma_{2}=(\sigma_{1}\sigma_{2})^{5}=T_{3,5}$.


$9_{18}=\sigma_{1}^{2}\cdot\sigma_{3}\sigma_{2}\sigma_{3}^{-1}\cdot\sigma_{3}\sigma_{2}\sigma_{1}^{2}\sigma_{2}^{-1}\sigma_{3}^{-1}\cdot\sigma_{2}\sigma_{1}\sigma_{2}^{-1}\cdot\sigma_{3}\nearrow_{2}\sigma_{1}^{2}\sigma_{3}\sigma_{2}^{2}\sigma_{1}^{3}\sigma_{3}\sim\sigma^{5}_{1}\sigma_{2}^{2}\sigma_{3}^{2}$\\
$\nearrow_{1}\sigma^{5}_{1}\sigma_{2}^{2}\sigma_{3}\sigma_{2}\sigma_{3}=\sigma^{5}_{1}\sigma_{2}^{3}\sigma_{3}\sigma_{2}\sim_{D}\sigma^{5}_{1}\sigma_{2}^{4}\sim\sigma_{2}\sigma_{1}^{5}\sigma_{2}^{3}\nearrow_{1}\sigma_{2}\sigma_{1}^{3}\sigma_{2}\sigma_{1}^{2}\sigma_{2}^{3}=\sigma_{2}\sigma_{1}(\sigma_{1}\sigma_{2})^{3}\sigma_{2}^{2}$\\
$=\sigma_{2}(\sigma_{1}\sigma_{2})^{4}\sigma_{2}=(\sigma_{1}\sigma_{2})^{5}=T_{3,5}$.


$9_{23}=\sigma_{1}^{2}\cdot\sigma_{3}\sigma_{2}\sigma_{3}^{-1}\cdot\sigma_{3}\sigma_{2}\sigma_{1}\sigma_{2}^{-1}\sigma_{3}^{-1}\cdot\sigma_{2}\sigma_{1}\sigma_{2}^{-1}\cdot\sigma_{3}^{2}\nearrow_{2}\sigma_{1}^{2}\sigma_{3}\sigma_{2}^{2}\sigma_{1}^{2}\sigma_{3}^{2}\sim\sigma_{1}^{4}\sigma_{2}^{2}\sigma_{3}^{3}$\\
$\nearrow_{1}\sigma_{1}^{4}\sigma_{2}^{2}\sigma_{3}\sigma_{2}\sigma_{3}^{2}=\sigma_{1}^{4}\sigma_{2}^{4}\sigma_{3}\sigma_{2}\sim_{D}\sigma_{1}^{4}\sigma_{2}^{5}\sim\sigma_{1}^{3}\sigma_{2}^{5}\sigma_{1}\nearrow_{1}\sigma_{1}^{3}\sigma_{2}^{2}\sigma_{1}\sigma_{2}^{3}\sigma_{1}=\sigma_{1}^{2}(\sigma_{1}\sigma_{2})^{3}\sigma_{2}\sigma_{1}$\\
$=\sigma_{1}(\sigma_{1}\sigma_{2})^{4}\sigma_{1}=(\sigma_{1}\sigma_{2})^{5}=T_{3,5}$.

$9_{38}=\sigma_{1}\cdot\sigma_{3}\sigma_{2}^{2}\sigma_{3}^{-1}\cdot\sigma_{2}\cdot\sigma_{2}\sigma_{1}\sigma_{2}^{-1}\cdot\sigma_{3}\cdot\sigma_{3}\sigma_{2}\sigma_{1}\sigma_{2}^{-1}\sigma_{3}^{-1}\nearrow_{3}\sigma_{1}\sigma_{3}\sigma_{2}^{4}\sigma_{1}\sigma_{3}^{2}\sigma_{2}\sigma_{1}\sigma_{3}^{-1}$\\
$=\sigma_{3}\sigma_{1}\sigma_{2}^{4}\sigma_{3}^{2}\sigma_{1}\sigma_{2}\sigma_{1}\sigma_{3}^{-1}\sim\sigma_{2}^{4}\sigma_{3}^{2}\sigma_{1}\sigma_{2}\sigma_{1}^{2}=\sigma_{2}^{4}\sigma_{3}^{2}\sigma_{2}^{2}\sigma_{1}\sigma_{2}\sim_{D}\sigma_{2}^{4}\sigma_{3}^{2}\sigma_{2}^{3}\nearrow_{1}\sigma_{2}^{2}\sigma_{3}\sigma_{2}^{2}\sigma_{3}^{2}\sigma_{2}^{3}$\\
$=\sigma_{2}(\sigma_{3}\sigma_{2})^{2}\sigma_{3}^{2}\sigma_{2}^{3}=(\sigma_{3}\sigma_{2})^{4}\sigma_{2}^{2}\sim\sigma_{2}(\sigma_{3}\sigma_{2})^{4}\sigma_{2}=(\sigma_{3}\sigma_{2})^{5}=T_{3,5}$.

$9_{49}=\sigma_{1}\cdot\sigma_{3}\sigma_{2}\sigma_{3}^{-1}\cdot\sigma_{1}^{2}\sigma_{2}\cdot\sigma_{2}\sigma_{1}\sigma_{2}^{-1}\cdot\sigma_{3}\nearrow_{2}\sigma_{1}\sigma_{3}\sigma_{2}\sigma_{1}^{2}\sigma_{2}^{2}\sigma_{1}\sigma_{3}\sim\sigma_{1}^{2}\sigma_{2}\sigma_{1}^{2}\sigma_{2}^{2}\sigma_{3}^{2}$\\
$\nearrow_{1}\sigma_{1}^{2}\sigma_{2}\sigma_{1}^{2}\sigma_{2}^{2}\sigma_{3}\sigma_{2}\sigma_{3}=\sigma_{1}^{2}\sigma_{2}\sigma_{1}^{2}\sigma_{2}^{3}\sigma_{3}\sigma_{2}\sim_{D}\sigma_{1}^{2}\sigma_{2}\sigma_{1}^{2}\sigma_{2}^{4}=(\sigma_{1}\sigma_{2})^{3}\sigma_{2}^{3}\nearrow_{1}(\sigma_{1}\sigma_{2})^{3}\sigma_{2}\sigma_{1}\sigma_{2}^{2}$\\
$=(\sigma_{1}\sigma_{2})^{5}=T_{3,5}$.

$10_{53}=\sigma_{1}\cdot\sigma_{3}\sigma_{2}\sigma_{3}^{-1}\cdot\sigma_{3}\sigma_{2}\sigma_{1}\sigma_{2}^{-1}\sigma_{3}^{-1}\cdot\sigma_{2}\sigma_{1}\sigma_{2}^{-1}\cdot\sigma_{4}\sigma_{3}^{2}\sigma_{4}^{-1}\cdot\sigma_{3}\sigma_{4}\nearrow_{3}\sigma_{1}\sigma_{3}\sigma_{2}^{2}\sigma_{1}^{2}\sigma_{4}\sigma_{3}^{3}\sigma_{4}$\\
$\sim\sigma_{4}\sigma_{3}\sigma_{4}\sigma_{2}^{2}\sigma_{1}^{3}\sigma_{3}^{3}=\sigma_{3}\sigma_{4}\sigma_{3}\sigma_{2}^{2}\sigma_{1}^{3}\sigma_{3}^{3}\sim_{D}\sigma_{3}^{2}\sigma_{2}^{2}\sigma_{1}^{3}\sigma_{3}^{3}\sim\sigma_{1}^{3}\sigma_{2}^{2}\sigma_{3}^{5}\nearrow_{1}\sigma_{1}\sigma_{2}\sigma_{1}^{2}\sigma_{2}^{2}\sigma_{3}^{5}$\\
$=\sigma_{2}^{2}\sigma_{1}\sigma_{2}\sigma_{2}^{2}\sigma_{3}^{5}\sim_{D}\sigma_{2}^{5}\sigma_{3}^{5}\nearrow_{1}\sigma_{2}^{5}\sigma_{3}\sigma_{2}\sigma_{3}^{4}=\sigma_{2}^{5}\sigma_{2}^{4}\sigma_{3}\sigma_{2}\sim_{D}\sigma_{2}^{10}=T_{2,10}$.

$10_{55}=\sigma^{2}_{1}\cdot\sigma_{3}\sigma_{2}\sigma^{-1}_{3}\cdot\sigma_{2}\sigma_{1}\sigma^{-1}_{2}\cdot\sigma_{4}\sigma_{3}^{2}\sigma^{-1}_{4}\cdot\sigma_{3}\sigma_{4}\nearrow_{3}\sigma^{2}_{1}\sigma_{3}\sigma_{2}^{2}\sigma_{1}\sigma_{4}\sigma^{3}_{3}\sigma_{4}\sim\sigma_{1}^{2}\sigma_{4}\sigma_{3}\sigma_{4}\sigma_{2}^{2}\sigma_{1}\sigma_{3}^{3}$\\
$=\sigma_{1}^{2}\sigma_{3}\sigma_{4}\sigma_{3}\sigma_{2}^{2}\sigma_{1}\sigma_{3}^{3}\sim_{D}\sigma_{1}^{2}\sigma_{3}^{2}\sigma_{2}^{2}\sigma_{1}\sigma_{3}^{3}\sim\sigma_{1}^{3}\sigma_{2}^{2}\sigma_{3}^{5}\nearrow_{1}\sigma_{1}\sigma_{2}\sigma_{1}^{2}\sigma_{2}^{2}\sigma_{3}^{5}=\sigma_{2}^{2}\sigma_{1}\sigma_{2}^{3}\sigma_{3}^{5}\sim_{D}\sigma_{2}^{5}\sigma_{3}^{5}$\\
$\nearrow_{1}\sigma_{2}^{5}\sigma_{3}\sigma_{2}\sigma_{3}^{4}=\sigma_{2}^{5}\sigma_{2}^{4}\sigma_{3}\sigma_{2}\sim_{D}\sigma_{2}^{10}=T_{2,10}$.

$10_{63}=\sigma^{2}_{1}\cdot\sigma_{4}\sigma_{3}\sigma_{2}\sigma_{1}\sigma^{-1}_{2}\sigma^{-1}_{3}\sigma^{-1}_{4}\cdot\sigma_{2}\cdot\sigma_{2}\sigma_{1}^{2}\sigma^{-1}_{2}\cdot\sigma_{3}\sigma_{4}\nearrow_{3}\sigma^{2}_{1}\sigma_{4}\sigma_{3}\sigma_{2}\sigma_{1}\sigma_{2}\sigma_{1}^{2}\sigma_{3}\sigma_{4}$\\
$\sim \sigma^{2}_{1}\sigma_{3}\sigma_{2}\sigma_{1}\sigma_{2}\sigma_{1}^{2}\sigma_{3}\sigma_{4}^{2}\nearrow_{1}\sigma^{2}_{1}\sigma_{3}\sigma_{2}\sigma_{1}\sigma_{2}\sigma_{1}^{2}\sigma_{3}\sigma_{4}\sigma_{3}\sigma_{4}=\sigma^{2}_{1}\sigma_{3}\sigma_{2}\sigma_{1}\sigma_{2}\sigma_{1}^{2}\sigma_{3}\sigma_{3}\sigma_{4}\sigma_{3}$\\
$\sim_{D}\sigma^{2}_{1}\sigma_{3}\sigma_{2}\sigma_{1}\sigma_{2}\sigma_{1}^{2}\sigma_{3}^{3}=\sigma^{2}_{1}\sigma_{3}\sigma_{1}\sigma_{2}\sigma_{1}\sigma_{1}^{2}\sigma_{3}^{3}=\sigma^{3}_{1}\sigma_{3}\sigma_{2}\sigma_{3}^{3}\sigma_{1}^{3}=\sigma^{3}_{1}\sigma_{2}^{3}\sigma_{3}\sigma_{2}\sigma_{1}^{3}\sim _{D}\sigma _{1}^{3}\sigma _{2}^{4}\sigma _{1}^{3}$\\
$\sim \sigma _{1}^{6}\sigma _{2}^{4} \nearrow_{1}\sigma _{1}^{6}\sigma _{2}\sigma _{1}\sigma _{2}^{3}=\sigma _{1}^{6}\sigma _{1}^{3}\sigma _{2}\sigma _{1}\sim _{D}\sigma _{1}^{10}=T_{2,10}$.


$10_{101}=\sigma_{1}^{2}\sigma_{2}\sigma_{3}\cdot\sigma_{3}\sigma_{2}\sigma_{1}\sigma_{2}^{-1}\sigma_{3}^{-1}\cdot\sigma_{2}\sigma_{1}\sigma_{2}^{-1}\cdot\sigma_{4}\cdot\sigma_{4}\sigma_{3}\sigma_{4}^{-1}\nearrow_{3}\sigma_{1}^{2}\sigma_{2}\sigma_{3}^{2}\sigma_{2}\sigma_{1}^{2}\sigma_{4}^{2}\sigma_{3}$\\
$=\sigma_{1}^{2}\sigma_{2}\sigma_{3}^{2}\sigma_{4}^{2}\sigma_{2}\sigma_{3}\sigma_{1}^{2}\sim\sigma_{1}^{4}\sigma_{2}\sigma_{3}^{2}\sigma_{4}^{2}\sigma_{2}\sigma_{3}\nearrow_{1}\sigma_{1}^{4}\sigma_{2}\sigma_{3}^{2}\sigma_{4}\sigma_{3}\sigma_{4}\sigma_{2}\sigma_{3}=\sigma_{1}^{4}\sigma_{2}\sigma_{3}^{3}\sigma_{4}\sigma_{3}\sigma_{2}\sigma_{3}$\\
$\sim_{D}\sigma_{1}^{4}\sigma_{2}\sigma_{3}^{4}\sigma_{2}\sigma_{3}=\sigma_{1}^{4}\sigma_{2}^{2}\sigma_{3}\sigma_{2}^{4}\sim_{D}\sigma_{1}^{4}\sigma_{2}^{6}\nearrow_{1}\sigma_{1}^{4}\sigma_{2}\sigma_{1}\sigma_{2}^{5}=\sigma_{1}^{9}\sigma_{2}\sigma_{1}\sim_{D}\sigma_{1}^{10}=T_{2,10}$.

$10_{120}=\sigma_{1}\cdot\sigma_{3}\sigma_{2}\sigma_{3}^{-1}\cdot\sigma_{4}\sigma_{3}\sigma_{2}\sigma_{1}\sigma_{2}^{-1}\sigma_{3}^{-1}\sigma_{4}^{-1}\cdot\sigma_{2}\cdot\sigma_{2}\sigma_{1}\sigma_{2}^{-1}\cdot\sigma_{4}\sigma_{3}\sigma_{4}^{-1}\cdot\sigma_{3}\sigma_{4}$\\
$\nearrow_{3}\sigma_{1}\sigma_{3}\sigma_{2}\sigma_{4}\sigma_{3}\sigma_{2}\sigma_{1}\sigma_{2}^{-1}\sigma_{4}^{-1}\sigma_{2}^{2}\sigma_{1}\sigma_{2}^{-1}\sigma_{4}\sigma_{3}^{2}\sigma_{4}=\sigma_{1}\sigma_{3}\sigma_{2}\sigma_{4}\sigma_{3}\sigma_{2}\sigma_{1}\sigma_{2}\sigma_{1}\sigma_{2}^{-1}\sigma_{3}^{2}\sigma_{4}$\\
$=\sigma_{1}\sigma_{3}\sigma_{2}\sigma_{4}\sigma_{3}\sigma_{2}^{2}\sigma_{1}\sigma_{3}^{2}\sigma_{4}\sim\sigma_{4}\sigma_{1}\sigma_{3}\sigma_{2}\sigma_{4}\sigma_{3}\sigma_{2}^{2}\sigma_{1}\sigma_{3}^{2}=\sigma_{1}\sigma_{4}\sigma_{3}\sigma_{4}\sigma_{2}\sigma_{3}\sigma_{2}^{2}\sigma_{1}\sigma_{3}^{2}$\\
$=\sigma_{1}\sigma_{3}\sigma_{4}\sigma_{3}\sigma_{2}\sigma_{3}\sigma_{2}^{2}\sigma_{1}\sigma_{3}^{2}\sim_{D}\sigma_{1}\sigma_{3}^{2}\sigma_{2}\sigma_{3}\sigma_{2}^{2}\sigma_{1}\sigma_{3}^{2}\sim\sigma_{3}^{4}\sigma_{2}\sigma_{3}\sigma_{2}^{2}\sigma_{1}^{2}=\sigma_{2}\sigma_{3}\sigma_{2}^{6}\sigma_{1}^{2}\sim_{D}\sigma_{2}^{7}\sigma_{1}^{2}$\\
$\nearrow_{1}\sigma_{2}^{5}\sigma_{1}\sigma_{2}^{2}\sigma_{1}^{2}=\sigma_{2}^{3}(\sigma_{2}\sigma_{1})^{3}\sigma_{1}=\sigma_{2}^{2}(\sigma_{2}\sigma_{1})^{4}\sim\sigma_{2}(\sigma_{2}\sigma_{1})^{4}\sigma_{2}=(\sigma_{2}\sigma_{1})^{5}=T_{3,5}$.

$10_{127}=\sigma^{3}_{1}\cdot\sigma^{-1}_{2}\sigma_{1}\sigma_{2}\cdot\sigma_{2}\sigma_{1}^{2}\sigma^{-1}_{2}\nearrow_{3}\sigma^{4}_{1}\sigma_{2}^{2}\sigma_{1}\sigma_{2}\sigma_{1}=\sigma^{4}_{1}\sigma_{1}\sigma_{2}\sigma_{1}^{2}\sigma_{1}\sim _{D}\sigma^{8}_{1}=T_{2,8}$.

$m(10_{145})=\sigma_{3}\sigma_{2}\cdot\sigma_{2}\sigma_{1}\sigma^{-1}_{2}\cdot\sigma^{2}_{3}\sigma_{2}\cdot\sigma_{2}\sigma_{1}\sigma^{-1}_{2}\nearrow_{3}\sigma_{3}\sigma_{2}\sigma_{1}\sigma_{2}\sigma_{1}\sigma^{2}_{3}\sigma^{2}_{2}\sigma_{1}=\sigma_{3}\sigma_{1}\sigma_{2}\sigma_{1}^{2}\sigma^{2}_{3}\sigma^{2}_{2}\sigma_{1}$\\
$=\sigma_{1}\sigma_{3}\sigma_{2}\sigma_{3}^{2}\sigma_{1}^{2}\sigma^{2}_{2}\sigma_{1}=\sigma_{1}\sigma_{2}^{2}\sigma_{3}\sigma_{2}\sigma_{1}^{2}\sigma^{2}_{2}\sigma_{1}\sim_{D}\sigma_{1}\sigma_{2}^{3}\sigma_{1}^{2}\sigma^{2}_{2}\sigma_{1}\nearrow_{5}(\sigma_{1}\sigma_{2})^{7}=T_{3,7}$.

$10_{149}=\sigma_{2}\sigma_{1}^{3}\sigma_{2}^{-1}\cdot\sigma_{2}^{-1}\sigma_{1}^{2}\sigma_{2}\cdot\sigma_{2}\nearrow_{2}\sigma_{2}\sigma_{1}^{5}\sigma_{2}^{2}\sim\sigma_{1}^{5}\sigma_{2}^{3}\nearrow_{1}\sigma_{1}^{5}\sigma_{2}\sigma_{1}\sigma_{2}^{2}=\sigma_{1}^{5}\sigma_{1}^{2}\sigma_{2}\sigma_{1}\sim_{D}\sigma_{1}^{8}$\\
$=T_{2,8}$.

$m(10_{157})=\sigma_{1}\cdot\sigma_{2}^{-1}\sigma_{1}^{2}\sigma_{2}\cdot\sigma_{2}\cdot\sigma_{2}\sigma_{1}^{2}\sigma_{2}^{-1}\nearrow_{2}\sigma_{1}^{3}\sigma_{2}^{3}\sigma_{1}^{2}\sim\sigma_{1}^{5}\sigma_{2}^{3}\nearrow_{1}\sigma_{1}^{5}\sigma_{2}\sigma_{1}\sigma_{2}^{2}=\sigma_{1}^{5}\sigma_{1}^{2}\sigma_{2}\sigma_{1}$\\
$\sim_{D}\sigma_{1}^{8}=T_{2,8}$.
\end{proof}

\begin{lemma}\label{cobordism4}
The quasi-positive representative of the knot $K$ in the knot table wears 

(1) a projective hat of degree-$6$ if $K$ is either $7_1$,  $8_{19}$, $9_3$, $9_6$, $9_{9}$, $9_{16}$, $10_{49}$,  $10_{66}$, $10_{80}$, $10_{128}$, $10_{134}$,  $10_{142}$, $10_{154}$ or $10_{161}$;

(2) a Hirzebruch hat of bidegree-$(3,3)$ if $K$ is either $7_1$,  $8_{19}$, $9_3$, $9_6$, $9_{9}$,      $10_{128}$ or $10_{134}$.


All these transverse knots have self-linking numbers  $5$.  
\end{lemma}

\begin{proof} We construct relative symplectic cobordisms as follows. The lemma follows from the same argument as that in the proof of Lemma~\ref{cobordism2}. 


$7_{1}=\sigma _{1}^{7}=T_{2,7}$.

$8_{19}=\sigma_{1}^{3}\sigma_{2}\sigma_{1}^{3}\sigma_{2}=\sigma_{1}^{2}\sigma_{2}\sigma_{1}\sigma_{2}\sigma_{1}^{2}\sigma_{2}\sim(\sigma_{1}\sigma_{2})^{2}\sigma_{1}^{2}\sigma_{2}\sigma_{1}=(\sigma_{1}\sigma_{2})^{4}=T_{3,4}$.



$9_{3}=\sigma_{1}^{6}\sigma_{2}\cdot\sigma^{-1}_{1}\sigma_{2}\sigma_{1}\nearrow_{1}\sigma_{1}^{6}\sigma_{2}^{2}\sigma_{1}\sim\sigma_{1}^{7}\sigma_{2}^{2}\nearrow_{1}\sigma_{1}^{5}\sigma_{2}\sigma_{1}^{2}\sigma_{2}^{2}=\sigma_{1}^{3}(\sigma_{1}\sigma_{2})^{3}\sigma_{2}=\sigma_{1}^{2}(\sigma_{1}\sigma_{2})^{4}$\\
$\sim\sigma_{1}(\sigma_{1}\sigma_{2})^{4}\sigma_{1}=(\sigma_{1}\sigma_{2})^{5}=T_{3,5}$.


$9_{6}=\sigma_{1}^{5}\sigma_{2}\cdot\sigma_{2}\sigma_{1}^{2}\sigma_{2}^{-1}\nearrow_{1}\sigma_{1}^{5}\sigma_{2}^{2}\sigma_{1}^{2}\sim\sigma_{1}^{7}\sigma_{2}^{2}\nearrow_{1}\sigma_{1}^{5}\sigma_{2}\sigma_{1}^{2}\sigma_{2}^{2}=\sigma_{1}^{3}(\sigma_{1}\sigma_{2})^{3}\sigma_{2}=\sigma_{1}^{2}(\sigma_{1}\sigma_{2})^{4}$\\
$\sim\sigma_{1}(\sigma_{1}\sigma_{2})^{4}\sigma_{1}=(\sigma_{1}\sigma_{2})^{5}=T_{3,5}$.


$9_{9}=\sigma_{1}^{4}\sigma_{2}\cdot\sigma_{2}\sigma_{1}^{3}\sigma_{2}^{-1}\nearrow_{1}\sigma_{1}^{4}\sigma_{2}^{2}\sigma_{1}^{3}\sim\sigma_{1}^{7}\sigma_{2}^{2}\nearrow_{1}\sigma_{1}^{5}\sigma_{2}\sigma_{1}^{2}\sigma_{2}^{2}=\sigma_{1}^{3}(\sigma_{1}\sigma_{2})^{3}\sigma_{2}=\sigma_{1}^{2}(\sigma_{1}\sigma_{2})^{4}$\\
$\sim\sigma_{1}(\sigma_{1}\sigma_{2})^{4}\sigma_{1}=(\sigma_{1}\sigma_{2})^{5}=T_{3,5}$.

$9_{16}=\sigma^{3}_{1}\sigma^{2}_{2}\cdot\sigma_{2}\sigma_{1}^{3}\sigma_{2}^{-1}\nearrow_{1}\sigma^{3}_{1}\sigma^{3}_{2}\sigma_{1}^{3}\sim \sigma^{6}_{1}\sigma^{3}_{2}\nearrow_{1}\sigma^{6}_{1}\sigma_{2}\sigma_{1}\sigma_{2}^{2}=\sigma^{6}_{1}\sigma_{1}^{2}\sigma_{2}\sigma_{1}\sim_{D}\sigma_{1}^{9}=T_{2,9}$.


$10_{49}=\sigma_{2}\sigma_{1}^{4}\sigma^{-1}_{2}\cdot\sigma_{1}\cdot\sigma_{3}\sigma_{2}^{2}\sigma^{-1}_{3}\cdot\sigma_{2}\sigma_{3}\nearrow_{2}\sigma_{2}\sigma^{5}_{1}\sigma_{3}\sigma^{3}_{2}\sigma_{3}\sim\sigma_{3}\sigma_{2}\sigma_{3}\sigma^{5}_{1}\sigma^{3}_{2}=\sigma_{2}\sigma_{3}\sigma_{2}\sigma^{5}_{1}\sigma^{3}_{2}$\\
$\sim_{D}\sigma_{2}^{2}\sigma^{5}_{1}\sigma^{3}_{2}\sim\sigma^{5}_{1}\sigma^{5}_{2}\nearrow_{1}\sigma^{5}_{1}\sigma_{2}\sigma_{1}\sigma_{2}^{4}=\sigma^{5}_{1}\sigma_{1}^{4}\sigma_{2}\sigma_{1}\sim_{D}\sigma^{10}_{1}=T_{2,10}$.

$10_{66}=\sigma_{1}^{3}\cdot\sigma_{3}\sigma_{2}\sigma_{1}\sigma_{2}^{-1}\sigma_{3}^{-1}\cdot\sigma_{2}\cdot\sigma_{2}\sigma_{1}^{2}\sigma_{2}^{-1}\cdot\sigma_{3}^{2}\nearrow_{2}\sigma_{1}^{3}\sigma_{3}\sigma_{2}\sigma_{1}\sigma_{2}\sigma_{1}^{2}\sigma_{3}^{2}\sim\sigma_{1}^{5}\sigma_{2}\sigma_{1}\sigma_{2}\sigma_{3}^{3}$\\
$\nearrow_{1}\sigma_{1}^{5}\sigma_{2}\sigma_{1}\sigma_{2}\sigma_{3}\sigma_{2}\sigma_{3}^{2}=\sigma_{1}^{5}\sigma_{2}\sigma_{1}\sigma_{2}\sigma_{2}^{2}\sigma_{3}\sigma_{2}\sim_{D}\sigma_{1}^{5}\sigma_{2}\sigma_{1}\sigma_{2}^{4}=\sigma_{1}^{5}\sigma_{1}^{4}\sigma_{2}\sigma_{1}\sim_{D}\sigma_{1}^{10}=T_{2,10}$.

$10_{80}=\sigma_{2}\sigma_{1}^{3}\sigma_{2}^{-1}\cdot\sigma_{1}^{2}\cdot\sigma_{3}\sigma_{2}^{2}\sigma_{3}^{-1}\cdot\sigma_{2}\sigma_{3}\nearrow_{2}\sigma_{2}\sigma_{1}^{5}\sigma_{3}\sigma_{2}^{3}\sigma_{3}\sim\sigma_{3}\sigma_{2}\sigma_{3}\sigma_{1}^{5}\sigma_{2}^{3}=\sigma_{2}\sigma_{3}\sigma_{2}\sigma_{1}^{5}\sigma_{2}^{3}$\\
$\sim_{D}\sigma_{2}^{2}\sigma_{1}^{5}\sigma_{2}^{3}\sim\sigma_{1}^{5}\sigma_{2}^{5}\nearrow_{1}\sigma_{1}^{5}\sigma_{2}\sigma_{1}\sigma_{2}^{4}=\sigma_{1}^{5}\sigma_{1}^{4}\sigma_{2}\sigma_{1}\sim_{D}\sigma_{1}^{10}=T_{2,10}$.


$10_{128}=\sigma^{3}_{1}\cdot\sigma_{3}\sigma_{2}\sigma^{-1}_{3}\cdot\sigma^{2}_{1}\sigma_{3}\sigma_{2}\cdot\sigma^{-1}_{3}\sigma_{2}\sigma_{3}=\sigma^{3}_{1}\sigma_{3}\sigma_{2}\sigma^{2}_{1}\sigma_{2}\sigma^{-1}_{3}\sigma_{2}\sigma_{3}\sim\sigma^{3}_{1}\sigma_{2}\sigma^{2}_{1}\sigma_{2}\sigma^{-1}_{3}\sigma_{2}\sigma_{3}^{2}$\\
$\nearrow_{2}\sigma^{3}_{1}\sigma_{2}\sigma^{2}_{1}\sigma_{2}^{2}\sigma_{3}\sigma_{2}\sigma_{3}=\sigma^{3}_{1}\sigma_{2}\sigma^{2}_{1}\sigma_{2}^{2}\sigma_{2}\sigma_{3}\sigma_{2}\sim_{D}\sigma^{3}_{1}\sigma_{2}\sigma^{2}_{1}\sigma_{2}^{4}=\sigma^{2}_{1}\sigma_{2}\sigma_{1}\sigma_{2}\sigma_{1}\sigma_{2}^{4}$\\
$=\sigma_{1}\sigma_{2}\sigma_{1}\sigma_{2}\sigma_{1}\sigma_{2}\sigma_{1}\sigma_{2}^{3}\sim\sigma_{2}\sigma_{1}\sigma_{2}\sigma_{1}\sigma_{2}\sigma_{1}\sigma_{2}\sigma_{1}\sigma_{2}^{2}=(\sigma_{1}\sigma_{2})^{5}=T_{3,5}$.

$10_{134}=\sigma^{3}_{1}\cdot\sigma_{3}\sigma_{2}\sigma^{-1}_{3}\cdot\sigma^{2}_{1}\sigma_{2}\sigma^{2}_{3}\nearrow_{1}\sigma^{3}_{1}\sigma_{3}\sigma_{2}\sigma^{2}_{1}\sigma_{2}\sigma^{2}_{3}\sim\sigma^{3}_{1}\sigma_{2}\sigma^{2}_{1}\sigma_{2}\sigma^{3}_{3}\nearrow_{1}\sigma^{3}_{1}\sigma_{2}\sigma^{2}_{1}\sigma_{2}\sigma_{3}\sigma_{2}\sigma_{3}^{2}$\\
$=\sigma^{3}_{1}\sigma_{2}\sigma^{2}_{1}\sigma_{2}\sigma_{2}^{2}\sigma_{3}\sigma_{2}\sim_{D}\sigma^{3}_{1}\sigma_{2}\sigma^{2}_{1}\sigma_{2}^{4}=\sigma^{2}_{1}\sigma_{2}\sigma_{1}\sigma_{2}\sigma_{1}\sigma_{2}^{4}=\sigma_{1}\sigma_{2}\sigma_{1}\sigma_{2}\sigma_{1}\sigma_{2}\sigma_{1}\sigma_{2}^{3}$\\
$\sim\sigma_{2}\sigma_{1}\sigma_{2}\sigma_{1}\sigma_{2}\sigma_{1}\sigma_{2}\sigma_{1}\sigma_{2}^{2}=(\sigma_{1}\sigma_{2})^{5}=T_{3,5}$.


$10_{142}=\sigma^{3}_{1}\cdot\sigma_{3}\sigma_{2}\sigma^{-1}_{3}\cdot\sigma^{3}_{1}\sigma_{2}\sigma_{3}\sim\sigma^{3}_{1}\sigma_{2}\sigma^{3}_{1}\sigma^{-1}_{3}\sigma_{2}\sigma_{3}^{2}\nearrow_{2}\sigma^{3}_{1}\sigma_{2}\sigma^{3}_{1}\sigma_{3}\sigma_{2}\sigma_{3}^{2}=\sigma^{3}_{1}\sigma_{2}\sigma^{3}_{1}\sigma_{2}^{2}\sigma_{3}\sigma_{2}$\\
$\sim_{D}\sigma^{3}_{1}\sigma_{2}\sigma^{3}_{1}\sigma_{2}^{3}=\sigma^{2}_{1}\sigma_{2}\sigma_{1}\sigma_{2}\sigma^{2}_{1}\sigma_{2}^{3}\nearrow_{4}(\sigma_{1}\sigma_{2})^{7}=T_{3,7}$.

$10_{154}=\sigma^{2}_{1}\sigma_{2}\cdot\sigma^{-1}_{1}\sigma_{2}\sigma_{1}\cdot\sigma_{3}\sigma^{3}_{2}\sigma_{3}\nearrow_{2}\sigma^{2}_{1}\sigma_{2}^{2}\sigma_{1}\sigma_{3}\sigma_{2}\sigma_{1}\sigma_{2}^{2}\sigma_{3}=\sigma^{2}_{1}\sigma_{2}^{2}\sigma_{1}\sigma_{3}\sigma_{1}^{2}\sigma_{2}\sigma_{1}\sigma_{3}$\\
$=\sigma^{2}_{1}\sigma_{2}^{2}\sigma_{1}^{3}\sigma_{3}\sigma_{2}\sigma_{3}\sigma_{1}=\sigma^{2}_{1}\sigma_{2}^{2}\sigma_{1}^{3}\sigma_{2}\sigma_{3}\sigma_{2}\sigma_{1}\sim_{D}\sigma^{2}_{1}\sigma_{2}^{2}\sigma_{1}^{3}\sigma_{2}^{2}\sigma_{1}\sim\sigma^{3}_{1}\sigma_{2}^{2}\sigma_{1}^{3}\sigma_{2}^{2}$\\
$\nearrow_{4}\sigma_{1}^{2}\sigma_{2}\sigma_{1}\sigma_{2}\sigma_{1}\sigma_{2}\sigma_{1}\sigma_{2}\sigma_{1}(\sigma_{1}\sigma_{2})^{2}=(\sigma_{1}\sigma_{2})^{7}=T_{3,7}$.


$10_{161}=\sigma^{3}_{1}\sigma_{2}\cdot\sigma^{-1}_{1}\sigma_{2}\sigma_{1}\cdot\sigma_{1}\sigma^{2}_{2}\nearrow_{6}(\sigma_{1}\sigma_{2})^{7}=T_{3,7}$.
\end{proof}

\begin{lemma}\label{cobordism5}
The quasi-positive representatives of the knots $9_1$,   $10_{139}$ and $10_{152}$ in the knot table all wear projective hats of degree-$6$. Their self-linking numbers are all $7$.
\end{lemma}
\begin{proof} It follows from the following constructions of relative symplectic cobordisms.

$9_{1}=\sigma _{1}^{9}=T_{2,9}$.

$10_{139}=\sigma^{4}_{1}\sigma_{2}\sigma^{3}_{1}\sigma^{2}_{2}=\sigma^{3}_{1}\sigma_{2}\sigma_{1}\sigma_{2}\sigma^{2}_{1}\sigma^{2}_{2}\sim\sigma^{2}_{1}\sigma_{2}\sigma_{1}\sigma_{2}\sigma^{2}_{1}\sigma^{2}_{2}\sigma_{1}\nearrow_{4}\sigma^{2}_{1}\sigma_{2}\sigma_{1}\sigma_{2}\sigma_{1}\sigma_{2}\sigma_{1}\sigma_{2}\sigma_{1}\sigma_{2}\sigma_{1}\sigma_{2}^{2}$\\
$=(\sigma_{1}\sigma_{2})^{7}=T_{3,7}$.

$10_{152}=\sigma_{1}^{3}\sigma_{2}^{2}\sigma_{1}^{2}\sigma_{2}^{3}\nearrow_{2}\sigma_{1}^{3}\sigma_{2}\sigma_{1}^{2}\sigma_{2}\sigma_{1}^{2}\sigma_{2}^{3}=\sigma_{1}^{3}\sigma_{2}\sigma_{1}\sigma_{2}\sigma_{1}\sigma_{2}\sigma_{1}\sigma_{2}^{3}\sim\sigma_{2}^{2}\sigma_{1}^{2}(\sigma_{1}\sigma_{2})^{4}$\\
$\nearrow_{2}\sigma_{2}^{2}\sigma_{1}\sigma_{2}^{2}\sigma_{1}(\sigma_{1}\sigma_{2})^{4}=\sigma_{2}\sigma_{1}\sigma_{2}\sigma_{1}\sigma_{2}\sigma_{1}(\sigma_{1}\sigma_{2})^{4}=(\sigma_{1}\sigma_{2})^{7}=T_{3,7}$.
\end{proof}


\begin{lemma}\label{cobordism6}
The quasi-positive representatives of the links with crossing number $\leq 11$ and nonzero nullity in the link table all wear projective hats of degree-$6$. Among these transverse links, the following all wear both Hirzebruch hats of bidegree-$(3,3)$ and projective hats of degree-$4$:\\ $L10n94\{1,0\}$, $m(L10n104\{1,0,0\})$, $m(L10n104\{1,1,0\})$, $m(L11n226\{0\})$,\\  $m(L11n381\{0,0\})$,  $m(L11n381\{0,1\})$ and $L11n428\{1,0\}$. 
\end{lemma}

\begin{proof} We construct relative symplectic cobordisms as follows. 

$m(L8n6\{0,0\})=\sigma_{2}^{2}\sigma_{3}\sigma_{1}\sigma_{2}\cdot\sigma_{3}^{-1}\sigma_{2}\sigma_{3}\cdot\sigma_{1}\nearrow_{2}\sigma_{2}^{2}\sigma_{3}\sigma_{1}\sigma_{2}\sigma_{3}\sigma_{2}\sigma_{3}\sigma_{1}=\sigma_{2}^{2}\sigma_{3}\sigma_{1}\sigma_{3}\sigma_{2}\sigma_{3}^{2}\sigma_{1}$\\
$=\sigma_{2}^{2}\sigma_{3}^{2}\sigma_{1}\sigma_{2}\sigma_{1}\sigma_{3}^{2}=\sigma_{2}^{2}\sigma_{3}^{2}\sigma_{2}\sigma_{1}\sigma_{2}\sigma_{3}^{2}\sim_{D}\sigma_{2}^{2}\sigma_{3}^{2}\sigma_{2}^{2}\sigma_{3}^{2}\nearrow_{4}(\sigma_{2}\sigma_{3})^{6}=T_{3,6}$.

$L8n6\{1,0\}=\sigma_{1}\sigma_{2}^{2}\sigma_{1}^{2}\sigma_{2}^{2}\sigma_{1}\nearrow_{4}(\sigma_{1}\sigma_{2})^{6}=T_{3,6}$.

$L8n8\{1,0,1\}=m(L8n8\{0,1,1\})=\sigma_{1}\sigma_{2}^{2}\sigma_{1}\sigma_{3}\sigma_{2}^{2}\sigma_{3}\nearrow_{1}\sigma_{1}\sigma_{2}^{2}\sigma_{1}\sigma_{3}\sigma_{2}\sigma_{1}\sigma_{2}\sigma_{3}$\\
$=\sigma_{1}\sigma_{2}^{2}\sigma_{1}\sigma_{3}\sigma_{1}\sigma_{2}\sigma_{1}\sigma_{3}=\sigma_{1}\sigma_{2}^{2}\sigma_{1}^{2}\sigma_{3}\sigma_{2}\sigma_{3}\sigma_{1}=\sigma_{1}\sigma_{2}^{2}\sigma_{1}^{2}\sigma_{2}\sigma_{3}\sigma_{2}\sigma_{1}\sim_{D}\sigma_{1}\sigma_{2}^{2}\sigma_{1}^{2}\sigma_{2}^{2}\sigma_{1}$\\
$\nearrow_{4}(\sigma_{1}\sigma_{2})^{6}=T_{3,6}$.

$m(L9n18\{0\})=\sigma_{1}\sigma_{2}^{3}\sigma_{1}\sigma_{2}^{3}\sigma_{1}=\sigma_{1}\sigma_{2}^{2}\sigma_{1}\sigma_{2}\sigma_{1}\sigma_{2}^{2}\sigma_{1}\nearrow_{3}(\sigma_{1}\sigma_{2})^{6}=T_{3,6}$.

$L9n18\{1\}=\sigma_{2}\cdot\sigma_{3}^{-1}\sigma_{2}\sigma_{3}\cdot\sigma_{1}\sigma_{2}\cdot\sigma_{3}^{-1}\sigma_{2}\sigma_{3}\cdot\sigma_{1}\nearrow_{3}\sigma_{2}^{2}\sigma_{3}\sigma_{1}\sigma_{2}\sigma_{3}\sigma_{2}\sigma_{3}\sigma_{1}=\sigma_{2}^{2}\sigma_{3}\sigma_{1}\sigma_{3}\sigma_{2}\sigma_{3}^{2}\sigma_{1}$\\
$=\sigma_{2}^{2}\sigma_{3}^{2}\sigma_{1}\sigma_{2}\sigma_{1}\sigma_{3}^{2}=\sigma_{2}^{2}\sigma_{3}^{2}\sigma_{2}\sigma_{1}\sigma_{2}\sigma_{3}^{2}\sim_{D}\sigma_{2}^{2}\sigma_{3}^{2}\sigma_{2}^{2}\sigma_{3}^{2}\nearrow_{4}(\sigma_{2}\sigma_{3})^{6}=T_{3,6}$.

$m(L9n19\{0\})=\sigma_{1}\sigma_{2}^{2}\sigma_{1}^{2}\sigma_{2}\cdot\sigma_{1}^{-1}\sigma_{2}\sigma_{1}\nearrow_{1}\sigma_{1}\sigma_{2}^{2}\sigma_{1}^{2}\sigma_{2}^{2}\sigma_{1}\nearrow_{4}(\sigma_{1}\sigma_{2})^{6}=T_{3,6}$.

$L9n19\{1\}=\sigma_{2}^{2}\sigma_{1}\cdot\sigma_{2}^{-1}\sigma_{1}\sigma_{2}\cdot\sigma_{2}\sigma_{1}^{2}\nearrow_{1}\sigma_{2}^{2}\sigma_{1}^{2}\sigma_{2}^{2}\sigma_{1}^{2}\nearrow_{4}(\sigma_{2}\sigma_{1})^{6}=T_{3,6}$.

$L10n91\{1,0\}=\sigma_{2}\cdot\sigma_{3}^{-1}\sigma_{2}\sigma_{3}\cdot\sigma_{1}^{2}\cdot\sigma_{3}^{-1}\sigma_{2}\sigma_{3}\cdot\sigma_{3}\sigma_{2}\sigma_{1}^{2}\nearrow_{2}\sigma_{2}^{2}\sigma_{3}\sigma_{1}^{2}\sigma_{2}\sigma_{3}^{2}\sigma_{2}\sigma_{1}^{2}=\sigma_{2}^{2}\sigma_{1}^{2}\sigma_{3}\sigma_{2}\sigma_{3}^{2}\sigma_{2}\sigma_{1}^{2}$\\
$=\sigma_{2}^{2}\sigma_{1}^{2}\sigma_{2}^{2}\sigma_{3}\sigma_{2}^{2}\sigma_{1}^{2}\sim_{D}\sigma_{2}^{2}\sigma_{1}^{2}\sigma_{2}^{4}\sigma_{1}^{2}\nearrow_{6}(\sigma_{2}\sigma_{1})^{8}=T_{3,8}$.

$m(L10n93\{0,0\})=\sigma_{1}^{2}\sigma_{2}^{2}(\sigma_{1}\sigma_{2})^{3}\nearrow_{2}(\sigma_{1}\sigma_{2})^{6}=T_{3,6}$.

$L10n93\{0,1\}=\sigma_{1}^{-1}\sigma_{2}\sigma_{1}\cdot\sigma_{3}^{-1}\sigma_{2}^{-1}\sigma_{1}\sigma_{2}\sigma_{3}\cdot\sigma_{1}^{-2}\sigma_{2}\sigma_{1}^{2}\cdot\sigma_{3}^{-1}\sigma_{2}^{-1}\sigma_{1}\sigma_{2}\sigma_{3}\cdot\sigma_{1}^{-2}\sigma_{2}\sigma_{1}^{2}$\\
$\nearrow_{8}\sigma_{1}^{-1}\sigma_{2}\sigma_{1}^{2}\sigma_{2}\sigma_{3}\sigma_{2}\sigma_{1}^{3}\sigma_{2}\sigma_{3}\sigma_{2}\sigma_{1}^{2}\sim\sigma_{1}\sigma_{2}\sigma_{1}^{2}\sigma_{2}\sigma_{3}\sigma_{2}\sigma_{1}^{3}\sigma_{2}\sigma_{3}\sigma_{2}=\sigma_{2}^{2}\sigma_{1}\sigma_{2}^{2}\sigma_{3}\sigma_{2}\sigma_{1}^{3}\sigma_{2}\sigma_{3}\sigma_{2}$\\
$=\sigma_{2}^{2}\sigma_{1}\sigma_{3}\sigma_{2}\sigma_{3}^{2}\sigma_{1}^{3}\sigma_{2}\sigma_{3}\sigma_{2}=\sigma_{2}^{2}\sigma_{3}\sigma_{1}\sigma_{2}\sigma_{1}^{3}\sigma_{3}^{2}\sigma_{2}\sigma_{3}\sigma_{2}=\sigma_{2}^{2}\sigma_{3}\sigma_{2}^{3}\sigma_{1}\sigma_{2}\sigma_{3}^{2}\sigma_{2}\sigma_{3}\sigma_{2}$\\
$\sim_{D}\sigma_{2}^{2}\sigma_{3}\sigma_{2}^{4}\sigma_{3}^{2}\sigma_{2}\sigma_{3}\sigma_{2}=(\sigma_{2}\sigma_{3})^{2}\sigma_{2}^{3}\sigma_{3}^{2}\sigma_{2}\sigma_{3}\sigma_{2}\nearrow_{2}(\sigma_{2}\sigma_{3})^{2}\sigma_{2}^{2}\sigma_{3}\sigma_{2}\sigma_{3}\sigma_{2}\sigma_{3}\sigma_{2}\sigma_{3}\sigma_{2}$\\
$=(\sigma_{2}\sigma_{3})^{7}=T_{3,7}$.

$L10n94\{0,0\}=\sigma_{1}\cdot\sigma_{2}^{-1}\sigma_{1}\sigma_{2}\cdot\sigma_{2}\sigma_{1}\cdot\sigma_{2}^{-1}\sigma_{1}\sigma_{2}\cdot\sigma_{2}\nearrow_{2}\sigma_{1}^{2}\sigma_{2}^{2}\sigma_{1}^{2}\sigma_{2}^{2}\nearrow_{4}(\sigma_{1}\sigma_{2})^{6}=T_{3,6}$.


$L10n94\{1,0\}=L10n94\{1,1\}=\sigma_{3}^{-1}\sigma_{2}\sigma_{3}\cdot\sigma_{1}^{-2}\sigma_{2}^{-1}\sigma_{3}\sigma_{2}\sigma_{1}^{2}\cdot\sigma_{1}^{-1}\sigma_{2}\sigma_{1}\nearrow_{4}\sigma_{2}\sigma_{3}^{2}\sigma_{2}\sigma_{1}\sigma_{2}\sigma_{1}$\\
$=\sigma_{2}\sigma_{3}^{2}\sigma_{2}^{2}\sigma_{1}\sigma_{2}\sim_{D}\sigma_{2}\sigma_{3}^{2}\sigma_{2}^{3}\nearrow_{2}\sigma_{2}\sigma_{3}\sigma_{2}\sigma_{3}\sigma_{2}\sigma_{2}\sigma_{3}\sigma_{2}=(\sigma_{2}\sigma_{3})^{4}=T_{3,4}$.

$m(L10n94\{0,1\})=\sigma_{2}\sigma_{3}^{2}\sigma_{2}\sigma_{3}\sigma_{1}\sigma_{2}\cdot\sigma_{3}^{-1}\sigma_{2}\sigma_{3}\cdot\sigma_{1}\nearrow_{2}\sigma_{2}\sigma_{3}^{2}\sigma_{2}\sigma_{3}\sigma_{1}\sigma_{2}\sigma_{3}\sigma_{2}\sigma_{3}\sigma_{1}$\\
$=\sigma_{2}\sigma_{3}^{2}\sigma_{2}\sigma_{3}\sigma_{1}\sigma_{3}\sigma_{2}\sigma_{3}^{2}\sigma_{1}=\sigma_{2}\sigma_{3}^{2}\sigma_{2}\sigma_{3}^{2}\sigma_{1}\sigma_{2}\sigma_{1}\sigma_{3}^{2}=\sigma_{2}\sigma_{3}^{2}\sigma_{2}\sigma_{3}^{2}\sigma_{2}\sigma_{1}\sigma_{2}\sigma_{3}^{2}$\\
$\sim_{D}\sigma_{2}\sigma_{3}^{2}\sigma_{2}\sigma_{3}^{2}\sigma_{2}^{2}\sigma_{3}^{2}=(\sigma_{2}\sigma_{3})^{3}\sigma_{2}^{2}\sigma_{3}^{2}\nearrow_{2}(\sigma_{2}\sigma_{3})^{6}=T_{3,6}$.

$m(L10n104\{0,0,0\})=\sigma_{1}\sigma_{2}\cdot\sigma_{1}^{-1}\sigma_{2}\sigma_{1}\cdot\sigma_{3}\sigma_{2}\cdot\sigma_{3}^{-1}\sigma_{2}\sigma_{3}\nearrow_{2}\sigma_{1}\sigma_{2}^{2}\sigma_{1}\sigma_{3}\sigma_{2}^{2}\sigma_{3}$\\
$\nearrow_{1}\sigma_{1}\sigma_{2}^{2}\sigma_{1}\sigma_{3}\sigma_{2}\sigma_{1}\sigma_{2}\sigma_{3}=\sigma_{1}\sigma_{2}^{2}\sigma_{1}\sigma_{3}\sigma_{1}\sigma_{2}\sigma_{1}\sigma_{3}=\sigma_{1}\sigma_{2}^{2}\sigma_{1}^{2}\sigma_{3}\sigma_{2}\sigma_{3}\sigma_{1}=\sigma_{1}\sigma_{2}^{2}\sigma_{1}^{2}\sigma_{2}\sigma_{3}\sigma_{2}\sigma_{1}$\\
$\sim_{D}\sigma_{1}\sigma_{2}^{2}\sigma_{1}^{2}\sigma_{2}^{2}\sigma_{1}\nearrow_{4}(\sigma_{1}\sigma_{2})^{6}=T_{3,6}$.


$m(L10n104\{1,0,0\})=m(L10n104\{0,0,1\})=\sigma_{2}^{-1}\sigma_{1}^{-2}\sigma_{2}\sigma_{1}^{2}\sigma_{2}\cdot\sigma_{3}\sigma_{2}^{-1}\sigma_{3}\sigma_{2}\sigma_{3}^{-1}$\\
$\nearrow_{5}\sigma_{1}^{2}\sigma_{2}\sigma_{3}\sigma_{2}\sigma_{3}\sigma_{2}=\sigma_{1}^{2}\sigma_{2}^{2}\sigma_{3}\sigma_{2}^{2}\sim_{D}\sigma_{1}^{2}\sigma_{2}^{4}\nearrow_{2}\sigma_{1}\sigma_{2}\sigma_{1}\sigma_{2}^{2}\sigma_{1}\sigma_{2}^{2}=(\sigma_{1}\sigma_{2})^{4}=T_{3,4}$.

$m(L10n104\{1,1,0\})=m(L10n104\{0,1,1\})=\sigma_{3}^{-1}\sigma_{2}\sigma_{1}^{-2}\sigma_{2}^{-1}\sigma_{3}\sigma_{2}\sigma_{1}^{2}\sigma_{2}^{-1}\sigma_{3}\cdot \sigma_{3}^{-1}\sigma_{2}^{-2}\sigma_{3}\sigma_{2}^{2}\sigma_{3}$\\
$\nearrow_{6}\sigma_{2}\sigma_{1}\sigma_{2}\sigma_{1}\sigma_{3}\sigma_{2}^{2}\sigma_{3}=\sigma_{2}^{2}\sigma_{1}\sigma_{2}\sigma_{3}\sigma_{2}^{2}\sigma_{3}\sim_{D}\sigma_{2}^{3}\sigma_{3}\sigma_{2}^{2}\sigma_{3}=\sigma_{2}(\sigma_{2}\sigma_{3})^{3}\nearrow_{1}(\sigma_{2}\sigma_{3})^{4}=T_{3,4}$.

$L10n104\{1,0,1\}=(\sigma_{1}\sigma_{2})^{2}\sigma_{1}\sigma_{3}\sigma_{2}\sigma_{3}\sigma_{2}\sigma_{3}=(\sigma_{1}\sigma_{2})^{3}\sigma_{3}\sigma_{2}^{2}\sigma_{3}\nearrow_{1}(\sigma_{1}\sigma_{2})^{3}\sigma_{3}\sigma_{2}\sigma_{1}\sigma_{2}\sigma_{3}$\\
$=(\sigma_{1}\sigma_{2})^{3}\sigma_{3}\sigma_{1}\sigma_{2}\sigma_{1}\sigma_{3}=(\sigma_{1}\sigma_{2})^{3}\sigma_{1}\sigma_{3}\sigma_{2}\sigma_{3}\sigma_{1}=(\sigma_{1}\sigma_{2})^{4}\sigma_{3}\sigma_{2}\sigma_{1}\sim_{D}(\sigma_{1}\sigma_{2})^{4}\sigma_{2}\sigma_{1}$\\
$\nearrow_{2}(\sigma_{1}\sigma_{2})^{4}\sigma_{2}\sigma_{1}\sigma_{2}^{2}=(\sigma_{1}\sigma_{2})^{6}=T_{3,6}$.

$m(L10n111\{0,0,1\})=\sigma_{2}\sigma_{3}^{2}\sigma_{4}\cdot\sigma_{3}^{-1}\sigma_{2}\sigma_{3}\cdot\sigma_{1}^{-1}\sigma_{2}\sigma_{1}\cdot\sigma_{3}\sigma_{4}\sigma_{1}\nearrow_{2}\sigma_{2}\sigma_{3}^{2}\sigma_{4}\sigma_{2}\sigma_{3}\sigma_{2}\sigma_{1}\sigma_{3}\sigma_{4}\sigma_{1}$\\
$=\sigma_{2}\sigma_{3}^{2}\sigma_{4}\sigma_{2}\sigma_{3}\sigma_{2}\sigma_{3}\sigma_{4}\sigma_{1}^{2}\nearrow_{1}\sigma_{2}\sigma_{3}^{2}\sigma_{4}\sigma_{2}\sigma_{3}\sigma_{2}\sigma_{3}\sigma_{4}\sigma_{1}\sigma_{2}\sigma_{1}=\sigma_{2}\sigma_{3}^{2}\sigma_{4}\sigma_{2}\sigma_{3}\sigma_{2}\sigma_{3}\sigma_{4}\sigma_{2}\sigma_{1}\sigma_{2}$\\
$\sim_{D}\sigma_{2}\sigma_{3}^{2}\sigma_{4}\sigma_{2}\sigma_{3}\sigma_{2}\sigma_{3}\sigma_{4}\sigma_{2}^{2}=\sigma_{2}\sigma_{3}^{2}\sigma_{4}\sigma_{2}^{2}\sigma_{3}\sigma_{2}\sigma_{4}\sigma_{2}^{2}=\sigma_{2}\sigma_{3}^{2}\sigma_{2}^{2}\sigma_{4}\sigma_{3}\sigma_{4}\sigma_{2}^{3}=\sigma_{2}\sigma_{3}^{2}\sigma_{2}^{2}\sigma_{3}\sigma_{4}\sigma_{3}\sigma_{2}^{3}$\\
$\sim_{D}\sigma_{2}\sigma_{3}^{2}\sigma_{2}^{2}\sigma_{3}^{2}\sigma_{2}^{3}\nearrow_{4}\sigma_{2}\sigma_{3}^{2}(\sigma_{2}\sigma_{3})^{5}\sigma_{2}=(\sigma_{2}\sigma_{3})^{7}=T_{3,7}$.

$m(L11n204\{0\})=(\sigma_{1}\sigma_{2})^{2}\sigma_{1}(\sigma_{1}\sigma_{2})^{3}=(\sigma_{1}\sigma_{2})^{5}\sigma_{1}\nearrow_{1}(\sigma_{1}\sigma_{2})^{6}=T_{3,6}$.

$L11n204\{1\}=\sigma_{3}^{-1}\sigma_{2}^{-1}\sigma_{1}\sigma_{2}\sigma_{3}\cdot\sigma_{1}^{-3}\sigma_{2}\sigma_{1}^{3}\cdot\sigma_{3}^{-1}\sigma_{2}\sigma_{3}\cdot\sigma_{1}^{-2}\sigma_{2}\sigma_{1}^{2}\nearrow_{9}\sigma_{1}\sigma_{2}\sigma_{3}\sigma_{2}\sigma_{1}^{3}\sigma_{2}\sigma_{3}\sigma_{2}\sigma_{1}\sigma_{2}\sigma_{1}$\\
$\sim\sigma_{2}\sigma_{3}\sigma_{2}\sigma_{1}^{3}\sigma_{2}\sigma_{3}\sigma_{2}\sigma_{1}\sigma_{2}\sigma_{1}^{2}=\sigma_{2}\sigma_{3}\sigma_{2}\sigma_{1}^{3}\sigma_{2}\sigma_{3}\sigma_{2}^{3}\sigma_{1}\sigma_{2}=\sigma_{2}\sigma_{3}\sigma_{2}\sigma_{1}^{3}\sigma_{3}^{3}\sigma_{2}\sigma_{3}\sigma_{1}\sigma_{2}$\\
$=(\sigma_{2}\sigma_{3})^{2}\sigma_{3}^{2}\sigma_{1}^{3}\sigma_{2}\sigma_{1}\sigma_{3}\sigma_{2}=(\sigma_{2}\sigma_{3})^{2}\sigma_{3}^{2}\sigma_{2}\sigma_{1}\sigma_{2}^{3}\sigma_{3}\sigma_{2}\sim_{D}(\sigma_{2}\sigma_{3})^{2}\sigma_{3}^{2}\sigma_{2}^{4}\sigma_{3}\sigma_{2}$\\
$=(\sigma_{2}\sigma_{3})^{2}\sigma_{3}^{2}\sigma_{2}^{3}\sigma_{3}\sigma_{2}\sigma_{3}\nearrow_{2}(\sigma_{2}\sigma_{3})^{2}\sigma_{3}\sigma_{2}\sigma_{3}\sigma_{2}\sigma_{3}\sigma_{2}(\sigma_{2}\sigma_{3})^{2}=(\sigma_{2}\sigma_{3})^{7}=T_{3,7}$.

$m(L11n205\{0\})=\sigma_{2}^{-2}\sigma_{1}\sigma_{2}^{2}\cdot\sigma_{1}\cdot\sigma_{2}^{-1}\sigma_{1}\sigma_{2}\cdot\sigma_{2}\sigma_{1}\nearrow_{3}\sigma_{1}\sigma_{2}^{2}\sigma_{1}^{2}\sigma_{2}^{2}\sigma_{1}\nearrow_{4}(\sigma_{1}\sigma_{2})^{6}=T_{3,6}$.

$L11n205\{1\}=\sigma_{2}\sigma_{3}^{3}\sigma_{2}\sigma_{3}\sigma_{1}\sigma_{2}\cdot\sigma_{3}^{-1}\sigma_{2}\sigma_{3}\cdot\sigma_{1}\nearrow_{2}\sigma_{2}\sigma_{3}^{3}\sigma_{2}\sigma_{3}\sigma_{1}\sigma_{2}\sigma_{3}\sigma_{2}\sigma_{3}\sigma_{1}$\\
$=\sigma_{2}\sigma_{3}^{3}\sigma_{2}\sigma_{3}\sigma_{1}\sigma_{3}\sigma_{2}\sigma_{3}^{2}\sigma_{1}=\sigma_{2}\sigma_{3}^{3}\sigma_{2}\sigma_{3}^{2}\sigma_{1}\sigma_{2}\sigma_{1}\sigma_{3}^{2}=\sigma_{2}\sigma_{3}^{3}\sigma_{2}\sigma_{3}^{2}\sigma_{2}\sigma_{1}\sigma_{2}\sigma_{3}^{2}\sim_{D}\sigma_{2}\sigma_{3}^{3}\sigma_{2}\sigma_{3}^{2}\sigma_{2}^{2}\sigma_{3}^{2}$\\
$=\sigma_{2}\sigma_{3}^{2}\sigma_{2}\sigma_{3}\sigma_{2}\sigma_{3}\sigma_{2}^{2}\sigma_{3}^{2}=(\sigma_{2}\sigma_{3})^{5}\sigma_{3}\nearrow_{1}(\sigma_{2}\sigma_{3})^{6}=T_{3,6}$.


$m(L11n226\{0\})=m(L11n226\{1\})=\sigma_{3}^{-1}\sigma_{2}^{-2}\sigma_{3}\sigma_{2}^{2}\sigma_{3}\cdot\sigma_{2}\sigma_{1}^{-2}\sigma_{2}^{-1}\sigma_{3}^{2}\sigma_{2}\sigma_{1}^{2}\sigma_{2}^{-1}\cdot\sigma_{2}\sigma_{1}^{-1}\sigma_{2}\sigma_{1}\sigma_{2}^{-1}$\\
$\nearrow_{4}\sigma_{2}^{2}\sigma_{3}^{3}\sigma_{2}\sigma_{1}\sigma_{2}\sigma_{1}\sigma_{2}^{-1}\sim\sigma_{2}\sigma_{3}^{3}\sigma_{2}\sigma_{1}\sigma_{2}\sigma_{1}=\sigma_{2}\sigma_{3}^{3}\sigma_{2}^{2}\sigma_{1}\sigma_{2}\sim_{D}\sigma_{2}\sigma_{3}^{3}\sigma_{2}^{3}\nearrow_{1}\sigma_{2}\sigma_{3}\sigma_{2}\sigma_{3}^{2}\sigma_{2}^{3}$\\
$=\sigma_{2}^{3}\sigma_{3}\sigma_{2}^{4}\sim_{D}\sigma_{2}^{7}=T_{2,7}$.

$m(L11n236\{0\})=\sigma_{2}\sigma_{3}^{3}\sigma_{2}\cdot\sigma_{3}^{-1}\sigma_{2}\sigma_{3}\cdot\sigma_{1}^{-1}\sigma_{2}\sigma_{1}\cdot\sigma_{3}^{2}\sigma_{1}\nearrow_{3}\sigma_{2}\sigma_{3}^{3}\sigma_{2}^{2}\sigma_{3}\sigma_{2}\sigma_{1}\sigma_{3}^{2}\sigma_{2}\sigma_{1}$\\
$=\sigma_{2}\sigma_{3}^{3}\sigma_{2}^{2}\sigma_{3}\sigma_{2}\sigma_{3}^{2}\sigma_{1}\sigma_{2}\sigma_{1}=\sigma_{2}\sigma_{3}^{3}\sigma_{2}^{2}\sigma_{3}\sigma_{2}\sigma_{3}^{2}\sigma_{2}\sigma_{1}\sigma_{2}\sim_{D}\sigma_{2}\sigma_{3}^{3}\sigma_{2}^{2}\sigma_{3}\sigma_{2}\sigma_{3}^{2}\sigma_{2}^{2}$\\
$\nearrow_{4}\sigma_{2}\sigma_{3}^{2}(\sigma_{2}\sigma_{3})^{4}\sigma_{3}(\sigma_{2}\sigma_{3})^{2}=(\sigma_{2}\sigma_{3})^{8}=T_{3,8}$.

$L11n237\{1\}=\sigma_{1}\sigma_{2}\sigma_{3}\cdot\sigma_{2}^{-1}\sigma_{3}\sigma_{2}\cdot\sigma_{1}^{-1}\sigma_{2}\sigma_{1}\cdot\sigma_{3}^{2}\cdot\sigma_{1}^{-1}\sigma_{2}\sigma_{1}\nearrow_{3}\sigma_{1}\sigma_{2}\sigma_{3}^{2}\sigma_{2}^{2}\sigma_{1}\sigma_{3}^{2}\sigma_{1}^{-1}\sigma_{2}\sigma_{1}\sigma_{2}$\\
$=\sigma_{1}\sigma_{2}\sigma_{3}^{2}\sigma_{2}^{2}\sigma_{3}^{2}\sigma_{2}\sigma_{1}\sigma_{2}\sim\sigma_{2}\sigma_{3}^{2}\sigma_{2}^{2}\sigma_{3}^{2}\sigma_{2}\sigma_{1}\sigma_{2}\sigma_{1}=\sigma_{2}\sigma_{3}^{2}\sigma_{2}^{2}\sigma_{3}^{2}\sigma_{2}^{2}\sigma_{1}\sigma_{2}\sim_{D}\sigma_{2}\sigma_{3}^{2}\sigma_{2}^{2}\sigma_{3}^{2}\sigma_{2}^{3}$\\
$\nearrow_{4}\sigma_{2}\sigma_{3}^{2}(\sigma_{2}\sigma_{3})^{5}\sigma_{2}=(\sigma_{2}\sigma_{3})^{7}=T_{3,7}$.

$m(L11n379\{0,1\})=(\sigma_{1}\sigma_{2})^{3}\sigma_{3}^{2}\sigma_{2}\sigma_{3}^{2}=(\sigma_{1}\sigma_{2})^{3}\sigma_{3}\sigma_{2}\sigma_{3}\sigma_{2}\sigma_{3}=(\sigma_{1}\sigma_{2})^{3}\sigma_{2}\sigma_{3}\sigma_{2}^{2}\sigma_{3}$\\
$\nearrow_{1}(\sigma_{1}\sigma_{2})^{3}\sigma_{2}\sigma_{3}\sigma_{2}\sigma_{1}\sigma_{2}\sigma_{3}=(\sigma_{1}\sigma_{2})^{3}\sigma_{2}\sigma_{3}\sigma_{1}\sigma_{2}\sigma_{1}\sigma_{3}=(\sigma_{1}\sigma_{2})^{3}\sigma_{2}\sigma_{1}\sigma_{3}\sigma_{2}\sigma_{3}\sigma_{1}$\\
$=(\sigma_{1}\sigma_{2})^{3}\sigma_{2}\sigma_{1}\sigma_{2}\sigma_{3}\sigma_{2}\sigma_{1}\sim_{D}(\sigma_{1}\sigma_{2})^{3}\sigma_{2}\sigma_{1}\sigma_{2}^{2}\sigma_{1}=(\sigma_{1}\sigma_{2})^{5}\sigma_{1}\nearrow_{1}(\sigma_{1}\sigma_{2})^{6}=T_{3,6}$.

$L11n379\{1,1\}=\sigma_{1}\sigma_{2}\cdot\sigma_{1}^{-1}\sigma_{2}\sigma_{1}\cdot\sigma_{2}^{-1}\sigma_{3}\sigma_{2}\cdot\sigma_{3}^{-1}\sigma_{2}\sigma_{3}\nearrow_{4}\sigma_{1}\sigma_{2}^{2}\sigma_{1}\sigma_{3}\sigma_{2}\sigma_{1}\sigma_{2}\sigma_{3}$\\
$=\sigma_{1}\sigma_{2}^{2}\sigma_{1}\sigma_{3}\sigma_{1}\sigma_{2}\sigma_{1}\sigma_{3}=\sigma_{1}\sigma_{2}^{2}\sigma_{1}^{2}\sigma_{3}\sigma_{2}\sigma_{3}\sigma_{1}=\sigma_{1}\sigma_{2}^{2}\sigma_{1}^{2}\sigma_{2}\sigma_{3}\sigma_{2}\sigma_{1}\sim_{D}\sigma_{1}\sigma_{2}^{2}\sigma_{1}^{2}\sigma_{2}^{2}\sigma_{1}$\\
$\nearrow_{4}(\sigma_{1}\sigma_{2})^{6}=T_{3,6}$.


$m(L11n381\{0,0\})=m(L11n381\{1,0\})=\sigma_{1}\cdot\sigma_{3}^{-1}\sigma_{2}^{-2}\sigma_{3}\sigma_{2}\sigma_{1}^{-1}\sigma_{2}\sigma_{1}\sigma_{2}^{-1}\sigma_{3}^{-1}\sigma_{2}^{2}\sigma_{3}\cdot\sigma_{1}$\\
$\nearrow_{4}\sigma_{1}\sigma_{2}^{2}\sigma_{1}\sigma_{2}\sigma_{3}\sigma_{1}\sim_{D}\sigma_{1}\sigma_{2}^{2}\sigma_{1}\sigma_{2}\sigma_{1}\nearrow_{2}(\sigma_{1}\sigma_{2})^{4}=T_{3,4}$.


$m(L11n381\{0,1\})=m(L11n381\{1,1\})=\sigma_{2}^{-1}\sigma_{1}^{-2}\sigma_{2}^{2}\sigma_{1}^{2}\sigma_{2}\cdot\sigma_{3}\sigma_{2}^{-1}\sigma_{3}\sigma_{2}\sigma_{3}^{-1}$\\
$\nearrow_{5}\sigma_{2}\sigma_{1}^{2}\sigma_{2}\sigma_{3}\sigma_{2}\sigma_{3}\sigma_{2}=\sigma_{2}\sigma_{1}^{2}\sigma_{2}^{2}\sigma_{3}\sigma_{2}^{2}\sim_{D}\sigma_{2}\sigma_{1}^{2}\sigma_{2}^{4}\nearrow_{1}\sigma_{2}\sigma_{1}^{2}\sigma_{2}^{2}\sigma_{1}\sigma_{2}^{2}=\sigma_{2}\sigma_{1}^{2}\sigma_{2}\sigma_{1}\sigma_{2}\sigma_{1}\sigma_{2}$\\
$=(\sigma_{2}\sigma_{1})^{4}=T_{3,4}$.

$m(L11n411\{0,0\})=\sigma_{1}\sigma_{2}\sigma_{3}^{2}\sigma_{4}\cdot\sigma_{3}^{-1}\sigma_{2}\sigma_{3}\cdot\sigma_{1}^{-1}\sigma_{2}\sigma_{1}\cdot\sigma_{3}\sigma_{4}\sigma_{1}\nearrow_{2}\sigma_{1}\sigma_{2}\sigma_{3}^{2}\sigma_{4}\sigma_{2}\sigma_{3}\sigma_{2}\sigma_{1}\sigma_{3}\sigma_{4}\sigma_{1}$\\
$\sim\sigma_{2}\sigma_{3}^{2}\sigma_{4}\sigma_{2}\sigma_{3}\sigma_{2}\sigma_{3}\sigma_{4}\sigma_{1}^{3}\nearrow_{1}\sigma_{2}\sigma_{3}^{2}\sigma_{4}\sigma_{2}\sigma_{3}\sigma_{2}\sigma_{3}\sigma_{4}\sigma_{1}\sigma_{2}\sigma_{1}^{2}=\sigma_{2}\sigma_{3}^{2}\sigma_{4}\sigma_{2}\sigma_{3}\sigma_{2}\sigma_{3}\sigma_{4}\sigma_{2}^{2}\sigma_{1}\sigma_{2}$\\
$\sim_{D}\sigma_{2}\sigma_{3}^{2}\sigma_{4}\sigma_{2}\sigma_{3}\sigma_{2}\sigma_{3}\sigma_{4}\sigma_{2}^{3}=\sigma_{2}\sigma_{3}^{2}\sigma_{4}\sigma_{2}^{2}\sigma_{3}\sigma_{2}\sigma_{4}\sigma_{2}^{3}=\sigma_{2}\sigma_{3}^{2}\sigma_{2}^{2}\sigma_{4}\sigma_{3}\sigma_{4}\sigma_{2}^{4}=\sigma_{2}\sigma_{3}^{2}\sigma_{2}^{2}\sigma_{3}\sigma_{4}\sigma_{3}\sigma_{2}^{4}$\\
$\sim_{D}\sigma_{2}\sigma_{3}^{2}\sigma_{2}^{2}\sigma_{3}^{2}\sigma_{2}^{4}\nearrow_{3}\sigma_{2}(\sigma_{2}\sigma_{3})^{2}\sigma_{2}^{2}\sigma_{3}^{2}\sigma_{2}^{2}\sigma_{3}\sigma_{2}^{2}=\sigma_{2}(\sigma_{2}\sigma_{3})^{2}\sigma_{2}^{2}\sigma_{3}^{2}(\sigma_{2}\sigma_{3})^{2}\sigma_{2}$\\
$=\sigma_{2}(\sigma_{2}\sigma_{3})^{2}\sigma_{2}(\sigma_{2}\sigma_{3})^{4}\sim(\sigma_{2}\sigma_{3})^{2}\sigma_{2}(\sigma_{2}\sigma_{3})^{4}\sigma_{2}=(\sigma_{2}\sigma_{3})^{7}=T_{3,7}$.

$L11n428\{1,0\}=L11n428\{1,1\}=\sigma_{1}\cdot\sigma_{2}^{-1}\sigma_{3}\sigma_{2}^{-1}\sigma_{3}^{-1}\sigma_{1}^{-1}\sigma_{2}\sigma_{1}\sigma_{3}\sigma_{2}\sigma_{3}^{-1}\sigma_{2}\cdot\sigma_{2}^{-2}\sigma_{3}\sigma_{2}^{2}$\\
$\nearrow_{3}\sigma_{2}\sigma_{1}\sigma_{3}\sigma_{2}^{3}\sim_{D}\sigma_{2}^{4}=T_{2,4}$.

$m(L11n432\{0,1\})=\sigma_{2}^{-1}\sigma_{3}^{2}\sigma_{2}\cdot\sigma_{3}^{3}\sigma_{2}\cdot\sigma_{3}^{-1}\sigma_{2}\sigma_{3}\cdot\sigma_{1}^{-1}\sigma_{2}\sigma_{1}\cdot\sigma_{1}\nearrow_{4}\sigma_{3}^{2}\sigma_{2}\sigma_{3}^{3}\sigma_{2}^{2}\sigma_{3}\sigma_{2}\sigma_{1}\sigma_{2}\sigma_{1}$\\
$=\sigma_{3}^{2}\sigma_{2}\sigma_{3}^{3}\sigma_{2}^{2}\sigma_{3}\sigma_{2}^{2}\sigma_{1}\sigma_{2}\sim_{D}\sigma_{3}^{2}\sigma_{2}\sigma_{3}^{3}\sigma_{2}^{2}\sigma_{3}\sigma_{2}^{3}=(\sigma_{3}\sigma_{2})^{2}\sigma_{3}(\sigma_{3}\sigma_{2})^{3}\sigma_{2}=(\sigma_{3}\sigma_{2})^{6}=T_{3,6}$.

$m(L11n433\{0,1\})=\sigma_{2}\sigma_{3}^{2}\cdot\sigma_{2}\sigma_{3}^{-1}\sigma_{4}\sigma_{3}\sigma_{2}^{-1}\cdot\sigma_{1}^{-1}\sigma_{2}\sigma_{1}\cdot\sigma_{3}\sigma_{4}\sigma_{1}\nearrow_{4}\sigma_{2}\sigma_{3}^{2}\sigma_{2}\sigma_{4}\sigma_{3}\sigma_{2}\sigma_{1}\sigma_{3}\sigma_{4}\sigma_{2}\sigma_{1}$\\
$=\sigma_{2}\sigma_{3}^{2}\sigma_{2}\sigma_{4}\sigma_{3}\sigma_{2}\sigma_{3}\sigma_{4}\sigma_{1}\sigma_{2}\sigma_{1}=\sigma_{2}\sigma_{3}^{2}\sigma_{2}\sigma_{4}\sigma_{2}\sigma_{3}\sigma_{2}\sigma_{4}\sigma_{2}\sigma_{1}\sigma_{2}=\sigma_{2}\sigma_{3}^{2}\sigma_{2}^{2}\sigma_{4}\sigma_{3}\sigma_{4}\sigma_{2}^{2}\sigma_{1}\sigma_{2}$\\
$=\sigma_{2}\sigma_{3}^{2}\sigma_{2}^{2}\sigma_{3}\sigma_{4}\sigma_{3}\sigma_{2}^{2}\sigma_{1}\sigma_{2}\sim_{D}\sigma_{2}\sigma_{3}^{2}\sigma_{2}^{2}\sigma_{3}^{2}\sigma_{2}^{3}\nearrow_{4}\sigma_{2}\sigma_{3}^{2}(\sigma_{2}\sigma_{3})^{5}\sigma_{2}=(\sigma_{2}\sigma_{3})^{7}=T_{3,7}$.

$L11n437\{1,0\}=\sigma_{1}\sigma_{2}^{2}\sigma_{1}\sigma_{3}\cdot\sigma_{2}^{-1}\sigma_{1}\sigma_{2}\cdot\sigma_{2}^{-1}\sigma_{3}\sigma_{2}\cdot\sigma_{2}\sigma_{1}\nearrow_{2}\sigma_{1}\sigma_{2}^{2}\sigma_{1}\sigma_{3}\sigma_{2}\sigma_{1}\sigma_{3}\sigma_{2}^{2}\sigma_{1}$\\
$=\sigma_{1}\sigma_{2}^{2}\sigma_{1}\sigma_{3}\sigma_{2}\sigma_{3}\sigma_{1}\sigma_{2}^{2}\sigma_{1}=\sigma_{1}\sigma_{2}^{2}\sigma_{1}\sigma_{2}\sigma_{3}\sigma_{2}\sigma_{1}\sigma_{2}^{2}\sigma_{1}\sim_{D}\sigma_{1}\sigma_{2}^{2}\sigma_{1}\sigma_{2}^{2}\sigma_{1}\sigma_{2}^{2}\sigma_{1}=(\sigma_{1}\sigma_{2})^{4}\sigma_{2}\sigma_{1}$\\
$\nearrow_{2}(\sigma_{1}\sigma_{2})^{6}=T_{3,6}$.

$L11n459\{1,1,1\}=\sigma_{2}\cdot\sigma_{3}^{-1}\sigma_{2}\sigma_{3}\cdot\sigma_{1}^{2}\cdot\sigma_{3}^{-1}\sigma_{2}\sigma_{3}\cdot\sigma_{3}^{2}\sigma_{2}\sigma_{1}^{2}\nearrow_{1}\sigma_{2}^{2}\sigma_{3}\sigma_{1}^{2}\sigma_{3}^{-1}\sigma_{2}\sigma_{3}^{3}\sigma_{2}\sigma_{1}^{2}$\\
$=\sigma_{2}^{2}\sigma_{1}^{2}\sigma_{2}\sigma_{3}^{3}\sigma_{2}\sigma_{1}^{2}\sim\sigma_{1}^{2}\sigma_{2}^{2}\sigma_{1}^{2}\sigma_{2}\sigma_{3}^{3}\sigma_{2}\nearrow_{1}\sigma_{1}^{2}\sigma_{2}\sigma_{3}\sigma_{2}\sigma_{1}^{2}\sigma_{2}\sigma_{3}^{3}\sigma_{2}=\sigma_{1}^{2}\sigma_{3}\sigma_{2}\sigma_{3}\sigma_{1}^{2}\sigma_{2}\sigma_{3}^{3}\sigma_{2}$\\
$=\sigma_{3}\sigma_{1}^{2}\sigma_{2}\sigma_{1}^{2}\sigma_{3}\sigma_{2}\sigma_{3}^{3}\sigma_{2}=\sigma_{3}\sigma_{2}\sigma_{1}\sigma_{2}^{2}\sigma_{1}\sigma_{3}\sigma_{2}\sigma_{3}^{3}\sigma_{2}\nearrow_{1}\sigma_{3}\sigma_{2}\sigma_{1}\sigma_{2}\sigma_{3}\sigma_{2}\sigma_{1}\sigma_{3}\sigma_{2}\sigma_{3}^{3}\sigma_{2}$\\
$=\sigma_{3}\sigma_{2}\sigma_{1}\sigma_{3}\sigma_{2}\sigma_{3}\sigma_{1}\sigma_{3}\sigma_{2}\sigma_{3}^{3}\sigma_{2}=\sigma_{3}\sigma_{2}\sigma_{3}\sigma_{1}\sigma_{2}\sigma_{1}\sigma_{3}^{2}\sigma_{2}\sigma_{3}^{3}\sigma_{2}=(\sigma_{3}\sigma_{2})^{2}\sigma_{1}\sigma_{2}\sigma_{3}^{2}\sigma_{2}\sigma_{3}^{3}\sigma_{2}$\\
$\sim_{D}(\sigma_{3}\sigma_{2})^{2}\sigma_{2}\sigma_{3}^{2}\sigma_{2}\sigma_{3}^{3}\sigma_{2}=(\sigma_{3}\sigma_{2})^{2}\sigma_{2}(\sigma_{3}\sigma_{2})^{2}\sigma_{3}^{2}\sigma_{2}=(\sigma_{3}\sigma_{2})^{6}=T_{3,6}$.

All these links are relative symplectic cobordant to $T_{3,11}$. So they all wear projective hats of degree-6. Moreover, by Lemma~\ref{cobordism1}, the following transverse links all wear both Hirzebruch hats of bidegree-$(3,3)$ and projective hats of degree-$4$: \\
$L10n94\{1,0\}$, $m(L10n104\{1,0,0\})$, $m(L10n104\{1,1,0\})$, $m(L11n226\{0\})$,\\  $m(L11n381\{0,0\})$,  $m(L11n381\{0,1\})$ and $L11n428\{1,0\}$. 
\end{proof}

\begin{remark}
One can also consider the quasi-positive transverse knots which are relative symplectic cobordant to $T_{5,6}$. Let $T(3,2;11,2)$ be the $(11,2)$-cable of the torus knot $T_{3,2}$, i.e., the $11\mu+2\lambda$ curve on the boundary of a tubular neighborhood of $T_{3,2}$, where $\mu$ is the meridian and $\lambda$ is the preferred longitude. Suppose $K$ is the transverse cable knot $T(3,2;11,2)$ which is presented by the closure of the braid $(\sigma_{1}\sigma_{2}\sigma_{3})^{5}\sigma_{1}\sigma_{2}$. Then $K$ is relative symplectic cobordant to $T_{5,6}$:

$(\sigma_{1}\sigma_{2}\sigma_{3})^{5}\sigma_{1}\sigma_{2}\nearrow_{2} (\sigma_{1}\sigma_{2}\sigma_{3})^{5}\sigma_{1}\sigma_{2}\sigma_{1}^{2} \sim_{S}(\sigma_{1}\sigma_{2}\sigma_{3})^{4}\sigma_{1}\sigma_{4}\sigma_{2}\sigma_{3}\sigma_{1}\sigma_{2}\sigma_{1}^{2}=\\
(\sigma_{1}\sigma_{2}\sigma_{3})^{4}\sigma_{1}\sigma_{4}\sigma_{2}\sigma_{3}(\sigma_{2}\sigma_{1})^2\nearrow_3 (\sigma_{1}\sigma_{2}\sigma_{3})^{4}\sigma_{1}\sigma_{4}\sigma_{2}\sigma_{3}\sigma_{2}\sigma_{4}\sigma_{1}\sigma_{2}\sigma_{3}\sigma_{4}\sigma_{1}=\\
(\sigma_{1}\sigma_{2}\sigma_{3})^{4}\sigma_{1}\sigma_{4}\sigma_{3}\sigma_{2}\sigma_{3}\sigma_{4}\sigma_{1}\sigma_{2}\sigma_{3}\sigma_{4}\sigma_{1}=
(\sigma_{1}\sigma_{2}\sigma_{3})^{4}\sigma_{4}\sigma_{3}(\sigma_{1}\sigma_{2}\sigma_{3}\sigma_{4})^2\sigma_{1}\nearrow_1 \\
\sigma_{1}\sigma_{2}\sigma_{3}\sigma_{1}\sigma_{2}\sigma_{4}\sigma_{3} (\sigma_{1}\sigma_{2}\sigma_{3})^{2}\sigma_{4}\sigma_{3}(\sigma_{1}\sigma_{2}\sigma_{3}\sigma_{4})^2\sigma_{1}=
\sigma_{2}\sigma_{3}\sigma_{4}(\sigma_{1}\sigma_{2}\sigma_{3}\sigma_{4})^5\sigma_{1}\\
\sim (\sigma_{1}\sigma_{2}\sigma_{3}\sigma_{4})^6=T_{5,6}.$

So $K$ wears a projective hat of degree-6. By Part (4) of Lemma~\ref{sig0}, $\sigma(K)=\sigma_{K}(-1)=\sigma_{T_{3,2}}(1)+\sigma_{T_{11,2}}(-1)=\sigma(T_{11,2})=-10$. On the other hand, $n-m=-13$. So by the same argument as in the proof of the 4th part of Theorem~\ref{3foldcover}, any non-negative definite exact filling $W$ of $\Sigma_{2}(K)$ is spin, has $b_{1}(W)=0$, $\chi(W)=15$ and $\sigma(W)=-10$. 
\end{remark}

\begin{remark}
There are infinitely many quasi-positive transverse links which are all relative symplectic cobordant to a fixed transverse knot $K$.  They can be constructed by the split unions of $K$ with transverse unlinks of any number of components, where each unknot component has the maximal self-linking number. Indeed, suppose a transverse link $L$ is relative symplectic cobordant to a transverse knot $K$,  and $L'$ is $L$ together with a maximal self-linking unknot that is separated from $L$ by a sphere, then $L'$ is relative symplectic cobordant to $L$, and hence to $K$. This is because the braid presentation of $L'$ is the same as that of $L$ with one more string, and one can apply \cite[Lemma 2.8]{eg} once and then do a destabilization to arrive at $L$.
\end{remark}

\end{document}